\documentclass[11pt]{scrartcl}
\pdfoutput=1

\newif\ifJSC

\JSCfalse %normal style

\usepackage[utf8]{inputenc}
\usepackage[T1]{fontenc}
\usepackage{lmodern}
\usepackage[english]{babel}

\usepackage{amssymb, amsmath, amsfonts, amsthm, color, mathtools}
\usepackage{stmaryrd} % llbracket, rrbracket
\usepackage{todonotes}
\usepackage{tikz}
\usepackage{graphicx}
\usepackage{url}
\usepackage{breakurl}

\usepackage{hyperref}
\usepackage{authblk} %For Multiple authors

\usepackage[normalem]{ulem} %for sout
\tikzstyle{myCol} = [black!30]

\newcommand{\R}{\mathbb{R}}
\newcommand{\K}{\mathbb{K}}

\newcommand{\Var}{\mathcal{V}}
\newcommand{\bgr}{{G \square H}}
\newcommand{\Bgr}{{\mathcal{G} \square \mathcal{H}}}
\newcommand{\G}{\mathcal{G}}
\renewcommand{\H}{\mathcal{H}}
\newcommand{\viz}{\text{viz}}
\renewcommand{\mod}{\mathrm{mod}~}

\newcommand{\Hfiber}[1]{\prescript{#1}{}{H}}
\newcommand{\Gfiber}[1]{G^{#1}}

\DeclareMathAlphabet{\mcal}{OMS}{cmsy}{m}{n}

\DeclarePairedDelimiter\set\{\}
\DeclarePairedDelimiter{\iverson}{\llbracket}{\rrbracket}

\newcommand{\transposed}{\top}

\newcommand{\revision}[1]{{\color{black}#1}}

\newtheorem{theorem}{Theorem}[section] %numbering within section
\newtheorem{definition}[theorem]{Definition}
\newtheorem{lemma}[theorem]{Lemma}
\newtheorem{observation}[theorem]{Observation}
\newtheorem{corollary}[theorem]{Corollary}
\newtheorem{conjecture}[theorem]{Conjecture}

\newtheorem{example}[theorem]{Example}
\newtheorem*{proposition*}{Proposition}
\newtheorem{remark}[theorem]{Remark}
\newtheorem*{example*}{Example}

\newcommand{\splitatcommas}[1]{%
	\begingroup
	\begingroup\lccode`~=`, \lowercase{\endgroup
		\edef~{\mathchar\the\mathcode`, \penalty0 \noexpand\hspace{0pt plus 1em}}%
	}\mathcode`,="8000 #1%
	\endgroup
}

\providecommand{\keywords}[1]{\textbf{Keywords:} #1}

\title{Sum-of-Squares Certificates for Vizing's Conjecture via 
Determining Gröbner Bases%
\thanks{This project has received funding from the Austrian 
	Science Fund (FWF): I\,3199-N31.
	Moreover, the second author has received funding from the Austrian Science 
	Fund (FWF): DOC~78.}
}

\author[1]{Elisabeth Gaar\ifJSC\footnote{Corresponding author}\fi}
\author[2]{Melanie Siebenhofer}
\affil[1]{Institute of Production and Logistics Management, Johannes Kepler 
University Linz, Austria, 
\href{mailto:elisabeth.gaar@jku.at}{elisabeth.gaar@jku.at}}
\affil[2]{Alpen-Adria-Universität Klagenfurt, 
	Austria, 
\href{mailto:melanie.siebenhofer@aau.at}{melanie.siebenhofer@aau.at}}
\date{}

\begin{document}

	\maketitle	
		
	\begin{abstract}
	The famous open Vizing conjecture claims 
	that the domination number of the Cartesian product graph of 
	two 
	graphs $G$ and $H$ is at least the product of the domination 
	numbers of~$G$ and~$H$.
	Recently Gaar, Krenn, Margulies and Wiegele used the graph class~$\G$ of 
	all graphs with~$n_\G$ vertices and domination number~$k_\G$ and   
	reformulated Vizing's 
	conjecture as the problem that 
	for all graph classes~$\G$ and~$\H$ the %so-called 
	Vizing 
	polynomial %~$f_\viz$ 
	is sum-of-squares (SOS) modulo the Vizing 
	ideal.
	By solving 
	semidefinite programs (SDPs) and clever guessing
	they %successfully 
	derived SOS-certificates for some values of 
	$k_\G$, $n_\G$, $k_\H$, and $n_\H$.  
	
	In this paper, we consider their approach for~$k_\G = k_\H = 1$. 
	For this case we are able to derive the unique reduced Gröbner 
	basis of the Vizing ideal.
	Based on this, we deduce the minimum degree $(n_\G + n_\H - 
	1)/2$ of an SOS-certificate for Vizing's conjecture, 
	which is the first result of this kind.
	Furthermore, we present a method to find certificates for graph 
	classes~$\G$ 
	and~$\H$ with $n_\G + n_\H -1 = d$ for general~$d$, which is again based on 
	solving SDPs, but does not depend on guessing and depends on much smaller 
	SDPs.
	We implement our new method in SageMath and give new 
	SOS-certificates for all graph classes~$\G$ and~$\H$ 
	with~$k_\G=k_\H=1$
	and~$n_\G + n_\H \leq 15$.
	
	\end{abstract}

	%\tableofcontents
	
	% sections
	\section{Introduction}
\label{chapter:intro}

A large area of graph theory focuses on the interrelationship of graph invariants.
One of these graph invariants is the domination number~$\gamma(G)$ of a simple undirected graph~$G$, that is the minimum size of a set of vertices in~$G$, such that each vertex in the graph is either in this set itself or adjacent to a vertex in this set.
In 1968, Vizing~\cite{vizing1968} made a conjecture regarding the domination 
number of the Cartesian product $\bgr$ of the graphs $G$ and $H$. 
The vertices of~$\bgr$ are the Cartesian product of the vertices in~$G$ and~$H$, 
and the subgraphs of~$\bgr$ induced by the vertices with same fixed first tuple 
entry are isomorphic to the graph~$G$ and analogously the vertices with the same 
second tuple entry induce subgraphs isomorphic to~$H$.
Vizing conjectured that for any graphs~$G$ and~$H$ it holds that~$\gamma(\bgr) 
\geq \gamma(G) \gamma(H)$.
To date, it is not clear whether this conjecture is true.
Nevertheless, for many classes of graphs it has already been shown that 
Vizing's conjecture holds, see the survey of Brešar et al.~\cite{survey2012} 
for details.

The first algebraic formulation of Vizing's conjecture has been done by 
Margulies and Hicks in~\cite{margulies_illya2012}. 
An algebraic method to solve combinatorial problems is to encode the problem as 
a system of polynomial equations and apply the Nullstellensatz or 
Positivstellensatz.
In several areas this and similar approaches have been used to show new 
results, for example for 
colorings~\cite{alon-tarsi1992,loera-jesus1995,loera-etal2008,shalom1992,hillar-windfeldt2008,lovasz1994,matiyasevich2001,mnuk2001},
 stable 
sets~\cite{loera-jesus1995,loera-etal2009,gouveia2010,li1981,lovasz1994,simis-wolmer-villarreal1994},
 flows~\cite{alon-tarsi1992,mnuk2001,onn2004} and matchings~\cite{fischer1988} 
in graphs.
Gaar, Krenn, Margulies and Wiegele~\cite{vizing-short2019} used an algebraic 
method to 
reformulate Vizing's 
conjecture as a sum-of-squares (SOS) program. 
In such a program one asks the 
question of whether it is possible to represent a non-negative polynomial as 
the sum of squares of polynomials.
SOS are heavily used in the area of polynomial optimization, see for 
example Blekherman, Parrilo and Thomas~\cite{blekherman-parrilo2013}, and also 
in many other fields like dynamical systems, geometric theorem proving and 
quantum mechanics, see for example 
Parrilo~\cite{parrilo2004}.

Such SOS programs can be solved with the help of semidefinite programming 
(SDP). 
Roughly speaking, a semidefinite program (SDP) is like a linear program but 
instead of a non-negative 
vector variable one has a positive semidefinite matrix variable and the 
Frobenius inner product is used instead of vector multiplications.
As for linear programs, there is also a duality theory for SDPs.
%The origin of SDPs can be located in 1963 in the paper of Bellman and 
%Fan~\cite{bellman-fan1963} on consistency conditions for systems of linear 
%inequalities (in the sense of positive semidefiniteness) in hermitian matrix 
%variables.
They are often used as relaxation\revision{s} of combinatorial optimization problems.
The first contribution to this area was the seminal paper of 
Lovász~\cite{lovasz1979} in 1979.
Around 1990, the interest in SDP exploded.
Nowadays there are several off-the-shelf solvers for SDPs, for example  
MOSEK~\cite{mosek} and SDPT3~\cite{sdpt3}.
Some nice survey papers on SDP are for example Vandenberghe and 
Boyd~\cite{vandenberghe-boyd1996} and 
Todd~\cite{todd2001}.

As already mentioned, Gaar et al.~\cite{vizing-short2019,vizing-long2020} 
presented a new approach for proving Vizing's conjecture by finding 
SOS Positivstellensatz certificates with the help of SDP.
In particular, they used SDP in order to prove that the so-called Vizing 
polynomial %~$f_\viz$ 
is SOS modulo the so-called Vizing 
ideal~$I_\viz$.
In addition, they provide code to computationally find numeric certificates and check certificates for correctness.
Furthermore, they gave certificates for the graph classes~$\G$ (all graphs 
with~$n_\G$ vertices and domination number~$k_\G$) and~$\H$ (all graphs 
with~$n_\H$ vertices and domination number~$k_\H$) with the property 
that~$n_\G$,~$k_\G$,~$n_\H$ and $k_\H$ satisfy~$k_\G = n_\G - 1 \geq 1$ 
and~$k_\H 
= n_\H - 1$ for~$n_\H \in \{2,3\}$ and the graph classes with~$k_\G = n_\G$ 
and~$k_\H = n_\H - d$ for~$d \leq 4$.

In this paper, we focus on the graph classes~$\G$ and~$\H$ with domination 
numbers~$k_\G = k_\H =1$.
Due to this special choice of the parameters, we are able to determine the 
unique reduced Gröbner basis of the Vizing ideal~$I_\viz$, which is an 
important part of finding SOS-certificates.
With the help of this Gröbner basis, we can determine the minimum 
degree of any certificate, which is~$(n_\G + n_\H - 1)/2$.

Furthermore, we show that if SOS-certificates are of a specific form, then the 
certificates for graph classes with~$n_\G + n_\H - 1 = d$ for a fixed 
integer~$d \geq 3$ depend on~$d$ only.
Based on that, we introduce a new method to find SOS-certificates for these 
graph 
classes. 
This method again makes use of SDP, 
but unlike in the approach of~\cite{vizing-short2019,vizing-long2020}, no 
algebraic numbers have to be guessed.
Additionally, the SDP that has to be solved is much smaller than the one in the 
other 
approach.
With the help of our implementation of the new algorithm in 
SageMath~\cite{sagemath}, we give certificates for all graph classes~$\G$ 
and~$\H$ with~$k_\G = k_\H = 1$ up to~$n_\G + n_\H \leq 15$.
\revision{
	The program code of the implementation discussed in Section~\ref{chapter:generalapproach}
	is available as ancillary files from the arXiv page of this paper
	at \href{https://arxiv.org/src/2112.04007/anc}{arxiv.org/src/2112.04007/anc}.
}

For \revision{these} specific graph classes with~$k_\G = k_\H =1$ it is clear 
that Vizing's 
conjecture holds, as it simply states that the domination number of the 
Cartesian product graph is greater or equal to $1$, which holds for every 
graph. Thus, in this paper we do not advance the knowledge on whether Vizing's 
conjecture is true for some graph classes or not. However, deriving new 
certificates via an algebraic method is an important step in the area of 
using conic linear optimization for computer-assisted proofs, because it 
demonstrates that deriving such proofs is possible for a wider set of graph 
classes.

The paper is structured as follows.
In Section~\ref{chapter:thbackground} we present all formal definitions and the 
background on the algebraic method of Gaar et al.\ we need for our results. 
In this paper, we focus on the case~$k_\G = k_\H = 1$.
First, we determine the reduced Gröbner basis and state the minimum degree of a 
certificate for~$k_\G = k_\H = 1$ in Section~\ref{chapter:gbasis}.
In Section~\ref{chapter:examples} we show  how to find 
2-SOS-certificates for~$n_\G=n_\H=2$ and for~$n_\G=3$ and $n_\H =2$.
Next, derived from the previous examples, we propose a general form of 
certificates for Vizing's conjecture in the case of~$k_\G=k_\H=1$ in 
Section~\ref{chapter:generalapproach}. We work out 
a new general method to prove the correctness of such certificates and we also 
list certificates for all graph classes with~$n_\G + n_\H \leq 15$, which we 
found using the newly implemented method.
Finally, we conclude and point out some open questions in 
Section~\ref{chapter:conclusion}.

	\section{Formal Definitions and Background}
\label{chapter:thbackground}

Vizing's conjecture is centered around the 
domination number defined as follows. 
%Note that for graph theory we follow the 
%notation of Diestel~\cite{diestel}.
\begin{definition}
	\label{def:dom-nr}
	Let $G = (V,E)$ be an undirected graph. A subset of vertices $D \subseteq 
	V(G)$ is a \emph{dominating set} of $G$ if for all $u \in V(G) \setminus D$ 
	there exists a vertex $v \in D$ such that $u$ is adjacent to $v$. In this 
	case we say that vertex $v$ \emph{dominates} vertex $u$. A dominating set 
	is a \emph{minimum dominating set} of $G$ if there is no dominating set 
	with smaller cardinality. The \emph{domination number} $\gamma(G)$ is 
	the cardinality of a minimum dominating set of $G$.
\end{definition}

Our interest lies in the behavior of the domination number on the product of 
two graphs, the so-called Cartesian product graph, which is defined as follows.
\begin{definition}
	%\label{def:cart-prodgraph}
	Let $G$ and $H$ be two graphs. The \emph{Cartesian product graph} $G 
	\square H $ is a graph with vertices $V(G \square H) = V(G) \times V(H)$ 
	and edge set
	\begin{align*}
	E(G \square H) = \Big\{ \big\{ (g,h), (g',h') \big\} \Big\vert \  & g = g' 
	\in V(G) \text{ and } \{h,h'\} \in E(H)  \text{, or}\\
	& h = h' \in V(H) \text{ and } \{g, g'\} \in E(G) \Big\}.
	\end{align*}
\end{definition}
For convenience we will further write $gh$ for a vertex $(g,h) \in V(G\square  
H)$. 
%The following example illustrates the concept of the Cartesian product graph.
%\begin{example}
Figure~\ref{fig:example-prod-graph} shows the Cartesian product graph of 
the cyclic graph $C_4$ and the linear graph $P_4$. 
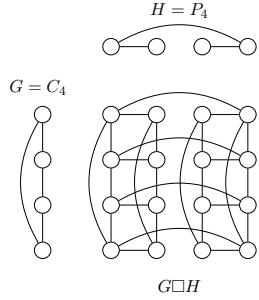
\begin{figure}[htbp]
	\centering
	\scalebox{0.6}%
	{		
		\begin{tikzpicture}[my node/.style={draw, circle}]
		\node[my node, fill = white] (4a) at (-1.5,0) {};
		\node[my node] (3a) at (-1.5,1) {};
		\node[my node, fill = white] (2a) at (-1.5,2) {};
		\node[my node] (1a) at (-1.5,3) {};
		
		\node[my node] (1b) at (0,4.5) {};
		\node[my node,  fill = white] (2b) at (1,4.5) {};
		\node[my node] (3b) at (2,4.5) {};
		\node[my node,  fill = white] (4b) at (3,4.5) {};
		
		\draw (1a) -- (2a) -- (3a) -- (4a); \draw (4a) to[bend left] (1a);
		\draw (1b) -- (2b); \draw (3b) -- (4b); \draw (4b) to[bend right] (1b);
		
		\node at (-1.6,3.6) {$G= C_4$};
		\node at (1.5,5.3) {$H=P_4$};
		
		\node[my node] (41) at (0,0) {};
		\node[my node] (42) at (1,0) {};
		\node[my node] (43) at (2,0) {};
		\node[my node] (44) at (3,0) {};
		
		\node[my node] (31) at (0,1) {};
		\node[my node] (32) at (1,1) {};
		\node[my node] (33) at (2,1) {};
		\node[my node] (34) at (3,1) {};
		
		\node[my node] (21) at (0,2) {};
		\node[my node] (22) at (1,2) {};
		\node[my node] (23) at (2,2) {};
		\node[my node] (24) at (3,2) {};
		
		\node[my node] (11) at (0,3) {};
		\node[my node] (12) at (1,3) {};
		\node[my node] (13) at (2,3) {};
		\node[my node] (14) at (3,3) {};

		\draw (11) -- (12); \draw (13) --(14); \draw (14) to[bend right] (11);
		\draw (21) -- (22); \draw (23) --(24); \draw (24) to[bend right] (21);
		\draw (31) -- (32); \draw (33) --(34); \draw (34) to[bend right] (31);
		\draw (41) -- (42); \draw (43) --(44); \draw (44) to[bend right] (41);
		
		\draw (11) -- (21) -- (31) --(41); \draw (41) to[bend left] (11);
		\draw (12) -- (22) -- (32) --(42); \draw (42) to[bend left] (12);
		\draw (13) -- (23) -- (33) --(43); \draw (43) to[bend left] (13);
		\draw (14) -- (24) -- (34) --(44); \draw (44) to[bend left] (14);
		\node at (1.5, -0.8) {$G \square H$};
		
		%		\node[my node,  fill = lightgray] at (1,0) {};
		%		\node[my node,  fill = lightgray]  at (1,2) {};
		%		\node[my node,  fill = lightgray]  at (3,1) {};
		%		\node[my node,  fill = lightgray]  at (3,3) {};

		%\foreach \angle[count=\i] in {0,15,...,75} \node[my node, shading 
		%angle=\angle] (a\i) at (0,\i) {Text};
		%\foreach \angle[count=\i] in {105,120,...,180} \node[my node, shading 
		%angle=\angle] (b\i) at (3,\i) {Text};
		%\path (a6) -- node[my node, shading angle=90]{Text} (b1);
		\end{tikzpicture}
	}
	\caption{The Cartesian product graph $C_4 \square P_4$}
	\label{fig:example-prod-graph}
\end{figure}
%\end{example}

For any vertex $g \in V(G)$ 
the 
vertices~$\revision{\{(g,h) \vert h \in V(H) \}}$ 
induce a subgraph of $\bgr$ that  
is isomorphic to~$H$.
Such a subgraph is called $H$-fiber and denoted by $\Hfiber{g}$.
Also for $h\in V(H)$ the subgraph $\Gfiber{h}$ of $\bgr$ induced 
by~$\revision{\{ 
(g,h) \vert g \in V(G)\}}$ is called a $G$-fiber.

In 1963, Vizing asked the question about the connection between the domination 
numbers of~$G$ and~$H$ and the domination number of the Cartesian product graph 
of~$G$ and~$H$ in~\cite{vizing1963}.
Five years later, he published the following conjecture in~\cite{vizing1968}.

\begin{conjecture}[Vizing's conjecture]
	%\label{conj:viz}
	For any two graphs $G$ and $H$, the inequality
	\begin{equation*}
	\gamma(G \square H) \geq \gamma(G) \gamma(H)
	\end{equation*}
	holds.
\end{conjecture}
To date, there is no answer to the question of whether Vizing's 
conjecture 
is true.
The typical approach to attack Vizing's conjecture is to show that for a 
specific $G$ Vizing's conjecture holds for any graph $H$.
Many results are based on the assumption that $G$ can be partitioned into 
subgraphs of a special kind.
The conjecture holds for example whenever~$G$ is a cycle, a 
tree\revision{,} or has domination number less \revision{than} or equal to~3 
and~$H$ may be any graph.
Furthermore, Zerbib~\cite{zerbib19} \revision{proved} in 2019 
that
\begin{equation*}
\gamma(\bgr) \geq \frac{1}{2} \gamma(G) \gamma(H) + \frac{1}{2}\max \{ 
\gamma(G),\gamma(H) \}\revision{,}
\end{equation*}
a weaker result\revision{.}  
To show that Vizing's conjecture is false, one may try to find a counterexample.
Some properties of a minimal counterexample are known, for example it has to be 
a graph with domination number greater than 
3 and for each vertex $g \in V(G)$ there has to exist a minimum dominating set 
that contains $g$. 
We refer to the survey paper of Brešar et al.~\cite{survey2012}
for an exceedingly nice and structured overview of the results on 
Vizing's conjecture.

In order to present new results for and with the approach of Gaar, Krenn, 
Margulies and Wiegele introduced in~\cite{vizing-short2019,vizing-long2020}, we 
continue 
with algebraic basics needed throughout the paper. \revision{For more details and, in particular,
the definitions of (total degree lexicographical) term orders, the leading term and (reduced) Gröbner bases,
which we will need in Section~\ref{chapter:gbasis}, we
refer the reader to the book of Cox, 
Little and O'Shea~\cite{cox-little-O'shea2015}.}

By $I$ we denote an ideal in a polynomial ring $P = \K[z_1, \dots, z_n]$ over a real field $\K \subseteq \R $.
By~$\overline{\K}$ we denote the algebraic closure of~$\K$.
\revision{We denote by $\Var(I) = \big \{ z^* \in \overline{\K}^n \, \big\vert \, f(z^*) = 0 \text{ for all } f \in I \big\} $ the variety of the ideal $I$.}
\revision{The ideal we consider in this paper is proven to be radical (i.e., 
for any 
	polynomial $f 
	\in P$ and any positive integer~$m$ the fact that $f^m \in I$ implies that 
	$f$ is in the ideal $I$) by Gaar et 
	al.~\cite{vizing-short2019}.}
%An example of a square-free polynomial in one variable $z$ is $z(z-1)$, as the 
%only non-constant divisors are $z$ and $z-1$, which are both not squares of a 
%non-constant polynomial. 
%This is an important fact to guarantee that the ideals we introduce later are 
%radical.
\revision{This allows us to apply the following important theorem.}
\begin{theorem}[Hilbert's Nullstellensatz \revision{for radical ideals}]
	\label{thm:hilbert}
	Let $P = \K [z_1,\dots,z_n]$ be a polynomial ring over a field $\K$ and $I \subseteq P$ a \revision{radical} ideal. 
	If $f(z^*) = 0$ for all $z^* \in \Var(I)$ for some $f \in P$, then $f$ is in the ideal $I$.
\end{theorem}

\revision{Note that} if the ideal~$I$ is finitely generated by the polynomials~$f_1, \dots, f_r \in P$, it is enough to check that $f$ vanishes on the common zeros of the generating polynomials (over the algebraic closure $\overline{\K}$).

The main idea of the approach by Gaar et al.\ is to prove Vizing's conjecture 
by showing for a particular constructed ideal that a specific polynomial is 
non-negative on the variety of the ideal.
For this purpose, they use the subsequent definitions.
\begin{definition}
	%\label{def:congruence}

%\end{definition}

%\begin{definition}
	%\label{def:sos}
	
	\revision{
		Two polynomials 
	$f, g \in P$ are \emph{congruent modulo an ideal}~$I$ (denoted by $f 
	\equiv g \ \ 
	\mod I$), if $f - g \in I$ or equivalently $f = g + h$ for some~$h \in 
	I$.	}
	
	Let $\ell$ be a non-negative integer. A polynomial~$f\in P$ is 
	$\ell$\emph{-sum-of-squares modulo} $I$ ($\ell$-SOS modulo $I$), if there 
	are polynomials $s_1, \dots , s_t \in P$ of degree at most $\ell$ such that
	\begin{equation*}
	f \equiv \sum\limits_{i=1}^t s_i^2 \ \ \mod I.
	\end{equation*}
	We say that the polynomials~$s_1, \dots, s_t$ form an \emph{SOS-certificate} 
	of degree~$\ell$.
\end{definition}

In the approach of Gaar et al., Vizing's conjecture is investigated for classes 
of graphs for~$G$ and $H$ with fixed number of vertices in the graph and fixed 
domination number.
The graph classes are denoted in the following way.
\begin{definition}
	\label{definition:graph-class}
	Let $n_\G$ and $k_\G$ be positive integers with $k_\G \leq n_\G$ defining the class of graphs $\G$ as the set of graphs with $n_\G$ vertices and fixed minimum dominating set $D_\G$ of size $k_\G$. 
\end{definition}

Without loss of generality, the minimum dominating set $D_\G$ is fixed. % in order to reduce the complexity of further considerations.
All other graphs can be obtained by relabeling the vertices.
For $k_\G = k$, we set $D_\G = \{g_1, \dots ,g_k \}$.
%Therefore, it suffices to show Vizing's conjecture on this special classes of 
%graphs.

In a next step, Gaar et al.\ construct an ideal, in which points in the variety 
correspond to graphs in the graph class $\G$.
The variables in this setting are boolean edge variables~$e_{gg'}$ indicating 
whether there is an edge between the vertices $g$ and $g'$.

\begin{definition}
	\label{def:ideal-graph}
	Let the set of variables be $e_\G = \{ e_{gg'} \ \vert \ g \neq g' \in 
	V(\G) \}.$ The ideal $I_\G \subseteq P_\G = \K[e_\G]$ is 
	\revision{generated} by the  
	polynomials
	\begin{subequations}\label{eqn:ideal-graph}
		\begin{align}
		e_{gg'}(e_{gg'} - 1) && \text{for all } \revision{g \neq g' \in 
		V(\G)},  \label{subeqn:edges} \\
		\prod\limits_{g'\in D_\G} (1- e_{gg'}) && \text{for all } g \in V(\G) \setminus D_\G, \label{subeqn:dom-edge}\\
		\prod\limits_{g' \in V(\G) \setminus S} \Big( \sum\limits_{g \in S}  e_{gg'} \Big) && \text{for all } S \subseteq V(\G) \text{ with } \lvert S \rvert = k_\G - 1. \label{subeqn:min-dom}
		\end{align}
	\end{subequations}
\end{definition}
\revision{Note that in the case $k_\G = 1$ \eqref{subeqn:dom-edge} 
simplifies 
to $(1- e_{gg_1})$ for all $g \in V(\G)$ with $g \neq g_1$ as $ D_\G = 
\{g_1\}$ 
and \eqref{subeqn:min-dom} is void.}

Gaar et al.\ proved that the following theorem holds.
\begin{theorem}%[\revision{Gaar et al.~\cite{vizing-short2019,vizing-long2020}}]
	\label{thm:bij-var-Ig}
	The points in the variety of $I_\G$ are in bijection to the graphs in $\G$.
\end{theorem}

%\begin{remark}
	We write $e_\G^*$ for elements in the variety of $I_\G$, and by $e_{gg'}^*$ 
	we denote the coordinate of $e_\G^*$ corresponding to the variable 
	$e_{gg'}$.
	
	Any point $e_\G^*$ in the variety of $I_\G$ is a common zero of the 
	generating polynomials. 
	Let $G$ be the graph the point  $e_\G^*$ corresponds to according to 
	Theorem~\ref{thm:bij-var-Ig}.
	Then~\eqref{subeqn:edges} ensures that $e_{gg'}^*$ is 
	either~0 or~1 and indicates whether $g$ is adjacent to $g'$ in 
	$G$,~\eqref{subeqn:dom-edge} guarantees that all vertices of $G$ are 
	dominated by $D_\G$ and~\eqref{subeqn:min-dom} makes sure that $D_\G$ is a 
	minimum 
	dominating set, thus the domination number of $G$ is indeed $k_\G$.
%\end{remark}

Analogously, Gaar et al.\ introduce the ideal $I_\H \subseteq P_\H = 
\K[e_\H]$ corresponding to the class $\H$, which contains all graphs of size 
$n_\H$ and fixed minimum dominating set $D_\H$ of size $k_\H$.
Moreover, they consider the graph class $\Bgr$, consisting of all Cartesian 
product graphs of graphs from $\G$ and $\H$.
The ideal $I_\Bgr$, where the boolean variables~$x_{gh}$ indicate whether the 
vertex $(g,h) \in V(\Bgr)$ is in the dominating set, 
is constructed as follows. 
\begin{definition}
	\label{def:ideal-prodgraph}
	Let $x_\Bgr = \set{ x_{gh} \ \vert \ g \in V(\G), h \in V(\H)}$. The ideal $I_\Bgr \subseteq P_\Bgr = \K[x_\Bgr \cup e_\G \cup e_\H]$ is generated by the polynomials
	\begin{subequations}\label{eqn:ideal-prodgraph}
		\begin{align}
		x_{gh}(x_{gh} - 1)&  \text{ and } \label{subeqn:vertices} \\
		%\end{align}
		%and
		%\begin{align}
		(1 - x_{gh}) \Bigg( \prod\limits_{\substack{g'\in V(\G) \\ g' \neq g }} (1 - e_{gg'}x_{g'h})  \Bigg)
		\Bigg( \prod\limits_{\substack{h'\in V(\H) \\ h' \neq h }} (1 - 
		e_{hh'}x_{gh'})  \Bigg)& \label{subeqn:dom-set}
		\end{align}
	\end{subequations}
	for all $g \in V(\G)$ and $h \in V(\H)$.
\end{definition}

Next, Gaar et al.\ introduced a final ideal with the following 
\revision{properties}.
\begin{definition}
	%\label{def:ideal-viz}
	For given graph classes $\G$ and $\H$ the \emph{Vizing 
	ideal}~$I_\viz \subseteq P_\Bgr$ is defined as the ideal generated by the 
	elements of 
	$I_\G$, $I_\H$ and $I_\Bgr$.
\end{definition}
%As Gaar et al.\ proved, the variety of~$I_\viz$ has the following property.
\begin{lemma}\label{lem:Ivizradicalandfinite}
	\revision{The ideal $I_\viz$ is radical with finite variety.}
\end{lemma}

\begin{theorem}
	\label{thm:bij-var-iviz}
	The points in the variety $\Var(I_\viz)$ are in bijection to the triples $(G, H, D)$, where~$G \in \G$, $H \in \H$ and $D \subseteq V(\bgr)$ is any (not necessary minimum) dominating set in~$\bgr$.
\end{theorem}

%\begin{remark}
	We denote the elements from the variety of $I_\viz$ by $z^*$, and by 
	$x_{gh}^*$, $e_{gg'}^*$ and $e_{hh'}^*$, we refer to the different 
	coordinates of $z^*$ for $g$, $g' \in V(\G)$, and $h$, $h' \in V(\H)$.
	For $z^* \in \Var(I_\viz)$, the polynomial~\eqref{subeqn:vertices} implies 
	that $x_{gh}^*$ is 0 or 1, which indicates if the vertex~$(g,h)$ is 
	in the dominating set $D$ that corresponds to $z^*$ according to 
	Theorem~\ref{thm:bij-var-iviz}.
	Furthermore,~\eqref{subeqn:dom-set} warrants that $D$ is a dominating set.
%\end{remark}

The polynomial of special interest for 
Gaar et al.\
is the so-called Vizing polynomial defined as follows.
\begin{definition}
	%\label{def:viz-poly}
	For given graph classes $\G$ and $\H$, the \emph{Vizing polynomial} is defined as
	\begin{equation*}
	f_\viz = \Big( \sum\limits_{gh \in V(\Bgr)} x_{gh} \Big) - k_\G k_\H.
	\end{equation*}
\end{definition}
With the help of this polynomial, Gaar et al.\ formulate the following 
important theorem, that provides a new method to prove Vizing's conjecture.
\begin{theorem}
	\label{cor:conjecture-sos}
	Vizing's conjecture is true if and only if for all positive integers 
	$n_\G$, $k_\G$, $n_\H$ and~$k_\H$ with $k_\G \leq n_\G$ and $k_\H \leq 
	n_\H$, there exists a  positive integer $\ell$ such that the Vizing 
	polynomial~$f_\viz$ is $\ell$-SOS modulo $I_\viz$.
\end{theorem}

\revision{
Note that Theorem~\ref{cor:conjecture-sos} is based on a result that connects 
the non-negativity of a polynomial on a variety of an ideal with the fact that 
this polynomial is $\ell$-SOS modulo the ideal for some value of $\ell$, see 
Gaar et al.~\cite[Lemma 2.8]{vizing-long2020}, and also Laurent~\cite[Theorem 
2.4]{laurent2009}, which is based on results by Parrilo~\cite{parrilo2002}, for 
further 
details.
}

\revision{
We want to point out that with arguments like the ones of 
Lasserre~\cite{lasserre2001}, one can obtain an upper bound on the $\ell$ to 
consider in Theorem~\ref{cor:conjecture-sos}. In particular, due to the 
generators~\eqref{subeqn:edges} and~\eqref{subeqn:vertices} of $I_\viz$, every 
monomial can be reduced over $I_\viz$ 
such that each variable has power at most one. Thus, when setting up the SDP, 
it suffices to consider all possible monomials that contain each variable with 
power at most one. As a result, in Theorem~\ref{cor:conjecture-sos} this 
gives $\ell \leq n_\G n_\H + \binom{n_\G}{2} + \binom{n_\H}{2}$.
}

To find SOS-certificates for Vizing's conjecture as in 
\revision{Theorem}~\ref{cor:conjecture-sos}, Gaar et 
al.~\cite{vizing-long2020} formulated these problems of finding 
SOS-certificates as SDPs as described below. 

They first fix $n_\G$, $n_\H$, $k_\G$ and $k_\H$ and determine 
$I_\viz$. 
Let $B$ be a Gröbner basis of $I_\viz$ and fix $\ell$ to be 
some positive integer.
Let $v$ be the vector of all monomials in $P_\Bgr$ of degree smaller or equal 
to $\ell$, that equal themselves when reduced by $B$. 
It is enough to consider these monomials as potential parts of the polynomials 
$s_1,\dots,s_t$ of an $\ell$-SOS-certificate.
Let~$u$ be the length of $v$.
Furthermore, let $S$ be a real $t \times u$ matrix, where the entries of 
row~$i$ represent the coefficients of the monomials from $v$ in $s_i$.
Then it holds that $Sv$ is the vector $(s_1,\dots,s_t)^\transposed$.
Now, let $X$ be the positive semidefinite matrix $S^\transposed S$, then 
\[\sum_{i=1}^{t} s_i^2 = (Sv)^\transposed(Sv) = v^\transposed Xv \]
holds.
As a result, the polynomials $s_1,\dots,s_t$ form an $\ell$-SOS-certificate if 
and only if
\[v^\transposed Xv \equiv f_\viz \quad \mod I_\viz\]
holds,
which is the case if for both sides of the equivalence the unique remainder of 
reduction by $B$ is the same.
% f \equiv g \mod I <==> f = g + h for h \in I <==> remainder of f and g is the same
By equating the coefficients, Gaar et 
al. obtain linear equations in the entries of the variable matrix~$X$.
To find a matrix $X$, which satisfies these equations and is additionally 
positive semidefinite, they set up an SDP with the constraints obtained by 
these 
equations.
The objective function of this SDP can be chosen arbitrarily as any feasible 
solution gives rise to an $\ell$-SOS-certificate.

\revision{Note that it is also possible to set up an SDP to decide whether a 
polynomial is $\ell$-SOS without the knowledge of a Gröbner basis, as it is 
described for example by Laurent~\cite{laurent2007, laurent2009}. However, the 
number 
of variables and constraints of this alternative SDP may be significantly 
larger. Thus, using the Gröbner basis of $I_\viz$ is a useful technical tool to 
reduce the size of the occurring SDP.}

Once an optimal solution~$X$ of the SDP is found, the matrix $S$ is derived by 
computing the eigenvalue decomposition $X = Q^\transposed \Lambda Q$ and setting $S = 
\Lambda^{1/2}Q$.
Unfortunately, the entries of $X$ are numerical, meaning that the values in $S$ do not represent an exact certificate.
The strategy of Gaar et al. is to find an objective function such that one can 
guess exact values for the entries in $X$ or $S$ and then check whether the 
obtained certificate is indeed valid with the code provided 
in~\cite{vizing-long2020}.

The final step is to prove the correctness of the found certificate algebraically. Ideally, one discovers some structures and finds a way to determine a general certificate for further graph classes like Gaar et al.\ did.

To sum up, the approach presented by Gaar et al.\ consists of the following 
steps. First 
fix $n_\G$, $n_\H$, $k_\G$ and $k_\H$ and compute a reduced Gröbner basis of 
$I_\viz$, 
then set up and solve an SDP in order to get a numeric certificate.
Next, guess an exact certificate 
and verify the certificate computationally. Finally, prove the correctness of 
the certificate and generalize the certificate. 
In this way, Gaar et al.\ successfully derived SOS-certificates for 
$k_\G = n_\G - 1 
\geq 1$ 
and~$k_\H = n_\H - 1$ where~$n_\H \in \{2,3\}$, and for~$k_\G = n_\G$ 
and~$k_\H = n_\H - d$ where~$d \leq 4$. 

	\section{Gröbner Basis of the Vizing Ideal 
	for~\texorpdfstring{$k_\G = k_\H = 1$}{kG=kH=1}}
\label{chapter:gbasis}

In this paper, we focus on Vizing's conjecture for graphs~$G$ and~$H$ with 
domination number 1, 
so 
we consider graph classes $\G$ and $\H$ with $k_\G = k_\H = 1$ and  
fixed dominating sets~$D_\G = \{g_1\}$ and~$D_\H = \{h_1\}$.
\revision{In particular, this implies that $g_1$ and $h_1$ are adjacent to all 
other 
vertices of $G$ and $H$, respectively.}
In this section, we first derive some simple statements with similar 
methods resulting from the work of Gaar et 
al.~\cite{vizing-short2019,vizing-long2020},
which we then use to determine the Gröbner basis of the Vizing ideal and to 
derive the minimum degree of any SOS-certificate.

\subsection{Auxiliary Results}
\label{sec:auxiliary-results}
The statements we derive in this section are 
based on the proof techniques of~\cite{vizing-short2019,vizing-long2020} and \revision{are} 
similar to those in Section 5.1 
of~\cite{vizing-long2020}.
We introduce the subsets of vertices
$$T_{gh} = \{ (g',h') \in V(\Bgr) \ \vert \ g' = g \text{ or }  h'=h \} = 
V(\G^{h}) \cup V(\prescript{g}{}{\H})$$ %For the sake of readability, we write 
%$T_{gh}$ for $T^{n_\G, n_\H}_{gh}$ whenever $n_\G$ and $n_\H$ are clear from 
%the context.
from 
$V(\Bgr)$, which are potentially adjacent to a vertex $(g,h)\in V(\Bgr)$ in 
$\Bgr$. % and depend on $n_\G$, $n_\H$, $g$ and $h$.
The corresponding variables of the vertices in $T_{gh}$ are exactly those 
which appear in the polynomial~\eqref{subeqn:dom-set}.
This leads to the 
following lemma.
\begin{lemma}
	\label{lemma:reduction-polynomial-in-ideal}
	The polynomial \[\prod_{(g',h') \in T_{gh}} (1 - x_{g'h'})\] is in the Vizing ideal $I_\viz$ for all $g \in V(\G)$, $h \in V(\H)$. 
\end{lemma}

\begin{proof}
	Let $z^* \in \Var(I_\viz)$ be a common zero of the generating polynomials of~$I_\viz$.
	This implies for given $g \in V(\G)$ and $h \in V(\H)$, that
	\begin{equation*}
	(1 - x^*_{gh}) \Bigg( \prod\limits_{\substack{g'\in V(\G) \\ g' \neq g }} (1 - e^*_{gg'}x^*_{g'h})  \Bigg)
	\Bigg( \prod\limits_{\substack{h'\in V(\H) \\ h' \neq h }} (1 - e^*_{hh'}x^*_{gh'})  \Bigg) = 0.
	\end{equation*}
	Moreover, we know that $e^*_{gg'} \in \{ 0, 1\}$ for all $g$, $g' \in 
	V(\G)$ and $e^*_{hh'} \in \{0, 1\}$ for all $h$, $h' \in V(\H)$.
	Therefore,  
	\begin{equation*}
	(1 - x^*_{gh}) \Bigg( \prod\limits_{\substack{g'\in V(\G) \\ g' \neq g }} (1 - x^*_{g'h})  \Bigg)
	\Bigg( \prod\limits_{\substack{h'\in V(\H) \\ h' \neq h }} (1 - x^*_{gh'})  \Bigg) = 0
	\end{equation*}
	holds, which implies that $z^*$ is a zero of $\prod_{(g',h') \in T_{gh}} (1 - x_{g'h'})$.
	Applying Hilbert's Nullstellensatz \revision{for radical ideals} 
	(Theorem~\ref{thm:hilbert}) proves the lemma, \revision{as $I_\viz$ is 
	radical (Lemma~\ref{lem:Ivizradicalandfinite})}.
\end{proof}

We further define the following polynomials.
\begin{definition}
	\label{def:rho}
	Let $(g,h) \in V(\Bgr)$ and let $i\leq n_\G + n_\H - 1 $ be a non-negative 
	integer.
	Then the polynomial $\rho_{gh}^i$ is defined as 
	\begin{equation*}
	\rho_{gh}^i = \sum_{\substack{S \subseteq T_{gh} \\ \lvert S \rvert = i}} \prod_{(g',h') \in S} x_{g'h'}.
	\end{equation*}
\end{definition}

Note that 
	the polynomial~$\rho_{gh}^i$ is the sum of all monomials consisting of $i$ 
	distinct variables from~$T_{gh}$.

\begin{lemma}
	\label{lemma:reduction-polynomial-as-sum}
	It holds that
	\begin{equation*}
	\prod_{(g',h') \in T_{gh}} (1 - x_{g'h'}) = \sum_{i = 0}^{n_\G + n_\H - 1}  (-1)^i  \rho_{gh}^i.
	\end{equation*}
\end{lemma}

\begin{proof}
	By expanding the product we get a sum whose summands are a product of $i$ 
	negative vertex variables and $\lvert T_{gh} \rvert - i = n_\G + n_\H - 1 - 
	i$ ones for all values of~$i$ between~$0$ and~$n_\G + n_H - 1$.
	Since the polynomial $\rho_{gh}^i$ is the sum of all monomials consisting of $i$ distinct variables corresponding to the vertices $(g',h') \in T_{gh}$ and $\rho_{gh}^0 = 1$, the equality holds.
\end{proof}

By simple reductions and combinatorial reasoning, we obtain the following lemma.
\begin{lemma}
	\label{lemma:product-two-polynomials-reduced}
	Let $T \subseteq V(\Bgr)$ be a non-empty set of cardinality $d$. For any positive integer~$k \leq d$ we define the polynomial $\pi^k$ as
	\begin{equation*}
	\pi^k = \sum_{\substack{S \subseteq T \\ \lvert S \rvert = k}} \prod_{(g,h) \in S} x_{gh}.
	\end{equation*}
	Then for all integers $i$, $j$ with $1 \leq i \leq j \leq d$ it holds that
	\begin{equation*}
	\pi^i \pi^j \equiv \sum_{r = 0}^{\min \{i, d - j\}} \binom{i}{r} \binom{j+r}{i} \pi^{j+r} \quad \mod I_{\viz}.
	\end{equation*}	
\end{lemma}

\begin{proof}
	From the generating polynomial~\eqref{subeqn:vertices} it follows that $x^2 \equiv x \ \mod I_\viz$ for all variables~$x \in T$.
	Furthermore, all monomials in the polynomial $\pi^j$ have degree $j$ and 
	those in $\pi^i$ are of degree $i$. This implies that the monomials in 
	$\pi^i \pi^j$ reduced by the 
	generating polynomials in~\eqref{subeqn:vertices} have at least degree~$j$ 
	and  
	the maximum degree is the minimum of $i + j$ 
	and the maximum number of distinct variables $d$.
	Therefore, 
	$$
\pi^i \pi^j \equiv \sum_{r = 0}^{\min \{i, d - j\}} \phi_r \pi^{j+r} \quad 
\mod I_{\viz}
$$	
	for some coefficients $\phi_r \in \mathbb{Z}$. 
		
	In order to determine $\phi_r$
	let us take a closer look at the coefficient of $\pi^{j+r}$ after we 
	reduced $\pi^i \pi^j$ by~\eqref{subeqn:vertices}. All 
	monomials of the polynomial~$\pi^{j+r}$ consist of $j+r$ different 
	variables.
	When we multiply two monomials $m_1$ and $m_2$ with $i$ and $j$ different 
	variables, the resulting reduced monomial $m$ consists of $j+r$ distinct 
	variables, if $r$ variables of $m_1$ are in~$m_1$ but not in $m_2$ and 
	$i-r$ variables of $m_1$ are in both monomials.
	This can be viewed as dividing $j+r$ variables into 3 groups with $i-r$, 
	$r$, and $j+r-i$ elements.
	Hence, for a fixed monomial $m$ in $\pi^{j+r}$, this gives us
	\begin{equation*}
	\phi_r = \frac{(j+r)!}{(i-r)!r!(j+r-i)!} = \binom{i}{r} \binom{j+r}{i}
	\end{equation*}
	different ways to choose the monomials $m_1$ and $m_2$ such that $m_1 m_2 
	\equiv m \ \mod I_\viz$.
\end{proof}

Lemma~\ref{lemma:product-two-polynomials-reduced} can be applied to the polynomials $\rho_{gh}^i$ and leads to the following corollary.

\begin{corollary}
	\label{cor:reduce-prod-polynomials}
	For all $(g,h) \in V(\Bgr)$ and for all integers $i$, $j$ with $1 \leq i \leq j \leq n_\G + n_\H - 1$ holds
	\begin{equation*}
	\rho^i_{gh} \rho^j_{gh} \equiv \sum_{r = 0}^{\min \{i, n_\G + n_\H - 1 - j\}} \binom{i}{r} \binom{j+r}{i} \rho^{j+r}_{gh} \quad \mod I_{\viz}.
	\end{equation*}	
\end{corollary}

One can observe that the form of products of such polynomials solely depends on 
$\lvert T_{gh} \rvert = n_\G + n_\H - 1$.
Later on, this fact allows us to derive the certificates for all graph classes 
$\G$ and~$\H$ with $k_\G = k_\H = 1$  and $n_\G + n_\H - 1 = d$ from the 
certificate of one of these graph classes for fixed $d$.

%\begin{remark}
%	Lemma~\ref{lemma:product-two-polynomials-reduced} also holds for
%	$ T \subseteq V(\Bgr) \cup \{ (g,g') \mid g \neq g' \in V(\G) \} \cup \{ (h,h') \mid h \neq h' \in V(\H) \} $ and
%	\begin{equation*}
%	\pi^k = \sum_{\substack{S \subseteq T \\ \lvert S \rvert = k}} \biggl( \prod_{(g,h) \in S \cap V(\Bgr)} x_{gh} \prod_{(g,g') \in S \cap V(\G)^2} e_{gg'} \prod_{(h,h') \in S \cap V(\H)^2} e_{hh'} \biggr)
%	\end{equation*}
%	as in the proof, we only made use of the fact that all variables are boolean.
%\end{remark}

\subsection{Gröbner Basis of the Vizing Ideal}
\label{sec:groebner-case1}
Our first step on the way to compute an SOS-certificate is to determine a 
Gröbner basis of the Vizing ideal $I_\viz$.
Note that in our case of $k_\G = k_\H = 1$ we fix the minimum dominating sets 
to $D_\G = \{g_1\}$ and $D_\H = \{h_1\}$.
We start with the following lemma. 

\begin{lemma}
	\label{lemma:1-e-in-ideal}
	Let $k_\G = 1$, then $1 - e_{gg_1} \in I_\G \subseteq I_\viz$ holds for all $g \in V(\G) \setminus \{g_1\}$.
\end{lemma}

\begin{proof}
	For $n_\G = 1$ this holds trivially, 
	for $n_\G > 1$ it follows directly from~\eqref{subeqn:dom-edge}. 
\end{proof}

To describe the elements of the Gröbner basis, we define the following sets of 
vertices.

\begin{definition}
	For some vertex $(g,h) \in V(\Bgr)$ in the Cartesian product graph we 
	define the following subsets of $T_{gh}$ as
	\begin{align*}
	U^r_{gh} & =
	\begin{dcases}
	\{ (g,h') \in V(\Bgr) \ \vert \ h' \notin \{h_1,h\} \}, & \text{for } h \neq h_1\\
	\emptyset, & \text{for } h = h_1
	\end{dcases},\\
	U^c_{gh} & = 
	\begin{dcases}
	\{ (g',h) \in V(\Bgr) \ \vert \ g' \notin \{g_1,g\} \}, & \text{for } g \neq g_1 \\
	\emptyset, & \text{for } g = g_1
	\end{dcases}, \\
	U_{gh} & = U^r_{gh} \cup U^c_{gh} & \text{and} \\
	\overline{U}_{gh} & = T_{gh} \setminus U_{gh}.		
	\end{align*}
\end{definition}
The next example uses different selections of $(g,h)$ to illustrate the rather technical definitions of $U_{gh}$ and $\overline{U}_{gh}$.
\begin{example}
	Figure~\ref{fig:example-Tgh-And-Ugh} shows the vertices in $U_{gh}$ and $\overline{U}_{gh}$ in one graph $\bgr$ of the graph class $\Bgr$ with $n_\G = n_\H = 4$ and $k_\G = k_\H = 1$ for different choices of $(g,h) \in V(\bgr)$.
	The vertex $(g,h)$ is highlighted with a thicker border, the vertices~\tikz\draw[fill=pink] (0,0) circle (.9ex); are in the set $\overline{U}_{gh}$, whereas the ones marked as~\tikz\draw[fill=teal] (0,0) circle (.9ex); are in~$U_{gh}$.
	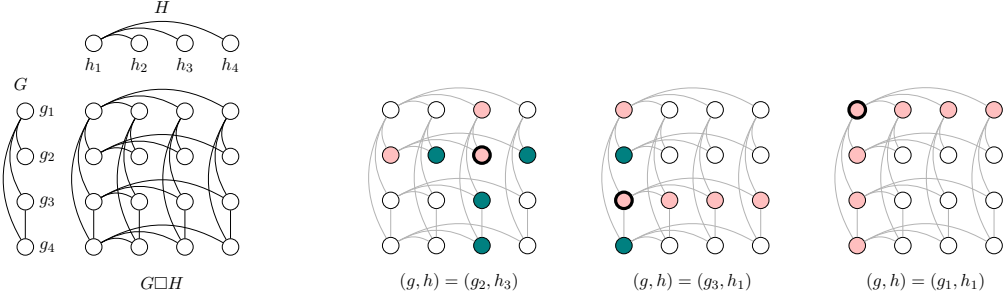
\begin{figure}[htbp]
				\begin{minipage}{0.37\textwidth}
			\centering
			\scalebox{0.6}%
			{		
				\begin{tikzpicture}[my node/.style={draw, circle}]
					\node[my node, fill = white] (4a) at (-1.5,0) {};
					\node[my node] (3a) at (-1.5,1) {};
					\node[my node, fill = white] (2a) at (-1.5,2) {};
					\node[my node] (1a) at (-1.5,3) {};
					\node[right=5pt] at (1a) {$g_1$}; 
					\node[right=5pt] at (2a) {$g_2$}; 
					\node[right=5pt] at (3a) {$g_3$}; 
					\node[right=5pt] at (4a) {$g_4$}; 
					
					\node[my node] (1b) at (0,4.5) {};
					\node[my node,  fill = white] (2b) at (1,4.5) {};
					\node[my node] (3b) at (2,4.5) {};
					\node[my node,  fill = white] (4b) at (3,4.5) {};
					\node[below=5pt] at (1b) {$h_1$};
					\node[below=5pt] at (2b) {$h_2$};
					\node[below=5pt] at (3b) {$h_3$};
					\node[below=5pt] at (4b) {$h_4$};
					
					\draw (1a) to[bend right](2a); \draw (1a) to[bend right] 
				(4a); \draw (1a) to[bend right] (3a); \draw (3a) -- (4a);
					\draw (1b) to[bend left](2b); \draw (1b) to[bend left] 
				(3b); \draw (1b) to[bend left] (4b);
					
					\node at (-1.6,3.6) {$G$};
					\node at (1.5,5.3) {$H$};
				
				\node[my node] (41) at (0,0) {};
				\node[my node] (42) at (1,0) {};
				\node[my node] (43) at (2,0) {};
				\node[my node] (44) at (3,0) {};
				
				\node[my node] (31) at (0,1) {};
				\node[my node] (32) at (1,1) {};
				\node[my node] (33) at (2,1) {};
				\node[my node] (34) at (3,1) {};
				
				\node[my node] (21) at (0,2) {};
				\node[my node] (22) at (1,2) {};
				\node[my node] (23) at (2,2) {};
				\node[my node] (24) at (3,2) {};
				
				\node[my node] (11) at (0,3) {};
				\node[my node] (12) at (1,3) {};
				\node[my node] (13) at (2,3) {};
				\node[my node] (14) at (3,3) {};
				
				% H Kanten
				\draw (11) to[bend left] (12); \draw (11) to[bend left](13); 
				\draw (11) to[bend left] (14);
				\draw (21) to[bend left] (22); \draw (21) to[bend left](23); 
				\draw (21) to[bend left] (24);
				\draw (31) to[bend left] (32); \draw (31) to[bend left](33); 
				\draw (31) to[bend left] (34);
				\draw (41) to[bend left] (42) ; \draw (41) to[bend left](43); 
				\draw (41) to[bend left] (44);
				
				\draw (11) to[bend right] (21); \draw (11) to[bend right](31); 
				\draw (11) to[bend right] (41); \draw (31)--(41);
				\draw (12) to[bend right] (22); \draw (12) to[bend right](32); 
				\draw (12) to[bend right] (42); \draw (32)--(42);
				\draw (13) to[bend right] (23); \draw (13) to[bend right](33); 
				\draw (13) to[bend right] (43); \draw (33)--(43);
				\draw (14) to[bend right] (24); \draw (14) to[bend right](34); 
				\draw (14) to[bend right] (44); \draw (34)--(44);
				\node at (1.5, -0.8) {$G \square H$};
				
				\end{tikzpicture}
			}
		\end{minipage}
		\begin{minipage}{0.2\textwidth}
			\centering
					\scalebox{0.6}%
			{		
			\begin{tikzpicture}[my node/.style={draw, circle}]
			%	\node[my node, fill = white] (4a) at (-1.5,0) {};
			%	\node[my node] (3a) at (-1.5,1) {};
			%	\node[my node, fill = white] (2a) at (-1.5,2) {};
			%	\node[my node] (1a) at (-1.5,3) {};
			%	
			%	\node[my node] (1b) at (0,4.5) {};
			%	\node[my node,  fill = white] (2b) at (1,4.5) {};
			%	\node[my node] (3b) at (2,4.5) {};
			%	\node[my node,  fill = white] (4b) at (3,4.5) {};
			%	
			%	\draw (1a) to[bend right](2a); \draw (1a) to[bend right] (4a); \draw (1a) to[bend right] (3a); \draw (3a) -- (4a);
			%	\draw (1b) to[bend left](2b); \draw (1b) to[bend left] (3b); \draw (1b) to[bend left] (4b);
			%	
			%	\node at (-1.6,3.6) {$G$};
			%	\node at (1.5,5.3) {$H$};
			
			\node at (1.5,5.3) {\phantom{$H$}};
			
			\node[my node] (41) at (0,0) {};
			\node[my node] (42) at (1,0) {};
			\node[my node,fill=teal] (43) at (2,0) {};
			\node[my node] (44) at (3,0) {};
			
			\node[my node] (31) at (0,1) {};
			\node[my node] (32) at (1,1) {};
			\node[my node,fill=teal] (33) at (2,1) {};
			\node[my node] (34) at (3,1) {};
			
			\node[my node,fill=pink] (21) at (0,2) {};
			\node[my node,fill=teal] (22) at (1,2) {};
			\node[my node,fill=pink,line width=2pt] (23) at (2,2) {};
			\node[my node,fill=teal] (24) at (3,2) {};
			
			\node[my node] (11) at (0,3) {};
			\node[my node] (12) at (1,3) {};
			\node[my node,fill=pink] (13) at (2,3) {};
			\node[my node] (14) at (3,3) {};
			
			% H Kanten
\draw[myCol] (11) to[bend left] (12); 
\draw[myCol] (11) to[bend left](13); 
\draw[myCol] (11) to[bend left] (14);
\draw[myCol] (21) to[bend left] (22); 
\draw[myCol] (21) to[bend left](23); 
\draw[myCol] (21) to[bend left] (24);
\draw[myCol] (31) to[bend left] (32); 
\draw[myCol] (31) to[bend left](33); 
\draw[myCol] (31) to[bend left] (34);
\draw[myCol] (41) to[bend left] (42) ; 
\draw[myCol] (41) to[bend left](43); 
\draw[myCol] (41) to[bend left] (44);

%G Kanten
\draw[myCol] (11) to[bend right] (21); 
\draw[myCol] (11) to[bend right](31); 
\draw[myCol] (11) to[bend right] (41); 
\draw[myCol] (31)--(41);
\draw[myCol] (12) to[bend right] (22); 
\draw[myCol] (12) to[bend right](32); 
\draw[myCol] (12) to[bend right] (42); 
\draw[myCol] (32)--(42);
\draw[myCol] (13) to[bend right] (23); 
\draw[myCol] (13) to[bend right](33); 
\draw[myCol] (13) to[bend right] (43); 
\draw[myCol] (33)--(43);
\draw[myCol] (14) to[bend right] (24); 
\draw[myCol] (14) to[bend right](34); 
\draw[myCol] (14) to[bend right] (44); 
\draw[myCol] (34)--(44);
			\node at (1.5, -0.8) {$(g,h) = (g_2,h_3)$};
			
			\end{tikzpicture}
		}
		\end{minipage}
		\begin{minipage}{0.2\textwidth}
			\centering
					\scalebox{0.6}%
			{		
			\begin{tikzpicture}[my node/.style={draw, circle}]
			
			%		\node[my node, fill = white] (4a) at (-1.5,0) {};
			%		\node[my node] (3a) at (-1.5,1) {};
			%		\node[my node, fill = white] (2a) at (-1.5,2) {};
			%		\node[my node] (1a) at (-1.5,3) {};
			%		
			%		\node[my node] (1b) at (0,4.5) {};
			%		\node[my node,  fill = white] (2b) at (1,4.5) {};
			%		\node[my node] (3b) at (2,4.5) {};
			%		\node[my node,  fill = white] (4b) at (3,4.5) {};
			%		
			%		\draw (1a) to[bend right](2a); \draw (1a) to[bend right] (4a); \draw (1a) to[bend right] (3a); \draw (3a) -- (4a);
			%		\draw (1b) to[bend left](2b); \draw (1b) to[bend left] (3b); \draw (1b) to[bend left] (4b);
			%		
			%		\node at (-1.6,3.6) {$G$};
			%		\node at (1.5,5.3) {$H$};
			
			\node at (1.5,5.3) {\phantom{$H$}};
			
			\node[my node,fill=teal] (41) at (0,0) {};
			\node[my node] (42) at (1,0) {};
			\node[my node] (43) at (2,0) {};
			\node[my node] (44) at (3,0) {};
			
			\node[my node,fill=pink, line width=2pt] (31) at (0,1) {};
			\node[my node, fill=pink] (32) at (1,1) {};
			\node[my node, fill=pink] (33) at (2,1) {};
			\node[my node, fill=pink] (34) at (3,1) {};
			
			\node[my node, fill=teal] (21) at (0,2) {};
			\node[my node] (22) at (1,2) {};
			\node[my node] (23) at (2,2) {};
			\node[my node] (24) at (3,2) {};
			
			\node[my node, fill=pink] (11) at (0,3) {};
			\node[my node] (12) at (1,3) {};
			\node[my node] (13) at (2,3) {};
			\node[my node] (14) at (3,3) {};
			
			% H Kanten
\draw[myCol] (11) to[bend left] (12); 
\draw[myCol] (11) to[bend left](13); 
\draw[myCol] (11) to[bend left] (14);
\draw[myCol] (21) to[bend left] (22); 
\draw[myCol] (21) to[bend left](23); 
\draw[myCol] (21) to[bend left] (24);
\draw[myCol] (31) to[bend left] (32); 
\draw[myCol] (31) to[bend left](33); 
\draw[myCol] (31) to[bend left] (34);
\draw[myCol] (41) to[bend left] (42) ; 
\draw[myCol] (41) to[bend left](43); 
\draw[myCol] (41) to[bend left] (44);

%G Kanten
\draw[myCol] (11) to[bend right] (21); 
\draw[myCol] (11) to[bend right](31); 
\draw[myCol] (11) to[bend right] (41); 
\draw[myCol] (31)--(41);
\draw[myCol] (12) to[bend right] (22); 
\draw[myCol] (12) to[bend right](32); 
\draw[myCol] (12) to[bend right] (42); 
\draw[myCol] (32)--(42);
\draw[myCol] (13) to[bend right] (23); 
\draw[myCol] (13) to[bend right](33); 
\draw[myCol] (13) to[bend right] (43); 
\draw[myCol] (33)--(43);
\draw[myCol] (14) to[bend right] (24); 
\draw[myCol] (14) to[bend right](34); 
\draw[myCol] (14) to[bend right] (44); 
\draw[myCol] (34)--(44);
			\node at (1.5, -0.8) {$(g,h)=(g_3,h_1)$};
			
			\end{tikzpicture}
		}
		\end{minipage}
		\begin{minipage}{0.2\textwidth} \centering	
					\scalebox{0.6}%
			{		
			\begin{tikzpicture}[my node/.style={draw, circle}]
			
			%	\node[my node, fill = white] (4a) at (-1.5,0) {};
			%	\node[my node] (3a) at (-1.5,1) {};
			%	\node[my node, fill = white] (2a) at (-1.5,2) {};
			%	\node[my node] (1a) at (-1.5,3) {};
			%	
			%	\node[my node] (1b) at (0,4.5) {};
			%	\node[my node,  fill = white] (2b) at (1,4.5) {};
			%	\node[my node] (3b) at (2,4.5) {};
			%	\node[my node,  fill = white] (4b) at (3,4.5) {};
			%	
			%	\draw (1a) to[bend right](2a); \draw (1a) to[bend right] (4a); \draw (1a) to[bend right] (3a); \draw (3a) -- (4a);
			%	\draw (1b) to[bend left](2b); \draw (1b) to[bend left] (3b); \draw (1b) to[bend left] (4b);
			%	
			%	\node at (-1.6,3.6) {$G$};
			%	\node at (1.5,5.3) {$H$};
			%	
			\node at (1.5,5.3) {\phantom{$H$}};
			
			\node[my node,fill=pink] (41) at (0,0) {};
			\node[my node] (42) at (1,0) {};
			\node[my node] (43) at (2,0) {};
			\node[my node] (44) at (3,0) {};
			
			\node[my node,fill=pink] (31) at (0,1) {};
			\node[my node] (32) at (1,1) {};
			\node[my node] (33) at (2,1) {};
			\node[my node] (34) at (3,1) {};
			
			\node[my node,fill=pink] (21) at (0,2) {};
			\node[my node] (22) at (1,2) {};
			\node[my node] (23) at (2,2) {};
			\node[my node] (24) at (3,2) {};
			
			\node[my node,fill=pink,line width=2pt] (11) at (0,3) {};
			\node[my node,fill=pink] (12) at (1,3) {};
			\node[my node,fill=pink] (13) at (2,3) {};
			\node[my node,fill=pink] (14) at (3,3) {};
			
			% H Kanten
			\draw[myCol] (11) to[bend left] (12); 
			\draw[myCol] (11) to[bend left](13); 
			\draw[myCol] (11) to[bend left] (14);
			\draw[myCol] (21) to[bend left] (22); 
			\draw[myCol] (21) to[bend left](23); 
			\draw[myCol] (21) to[bend left] (24);
			\draw[myCol] (31) to[bend left] (32); 
			\draw[myCol] (31) to[bend left](33); 
			\draw[myCol] (31) to[bend left] (34);
			\draw[myCol] (41) to[bend left] (42) ; 
			\draw[myCol] (41) to[bend left](43); 
			\draw[myCol] (41) to[bend left] (44);
			
			%G Kanten
			\draw[myCol] (11) to[bend right] (21); 
			\draw[myCol] (11) to[bend right](31); 
			\draw[myCol] (11) to[bend right] (41); 
			\draw[myCol] (31)--(41);
			\draw[myCol] (12) to[bend right] (22); 
			\draw[myCol] (12) to[bend right](32); 
			\draw[myCol] (12) to[bend right] (42); 
			\draw[myCol] (32)--(42);
			\draw[myCol] (13) to[bend right] (23); 
			\draw[myCol] (13) to[bend right](33); 
			\draw[myCol] (13) to[bend right] (43); 
			\draw[myCol] (33)--(43);
			\draw[myCol] (14) to[bend right] (24); 
			\draw[myCol] (14) to[bend right](34); 
			\draw[myCol] (14) to[bend right] (44); 
			\draw[myCol] (34)--(44);
			\node at (1.5, -0.8) {$(g,h)=(g_1,h_1)$};
			
			\end{tikzpicture}
		}
		\end{minipage}
		\caption{Illustration of  \revision{$G \square H$} and $U_{gh}$ and 
		$\overline{U}_{gh}$ for different 
		choices of $(g,h)$ \revision{in $G \square H$}}
		\label{fig:example-Tgh-And-Ugh}	
	\end{figure}
	
\end{example}

%\begin{remark}
	Note that the vertices in $U_{gh}$ and  $\overline{U}_{gh}$  are independent of the choice of $\bgr$ and only depend on $\G$ and $\H$.
	More precisely, the vertices in $\overline{U}_{gh}$ are exactly those adjacent to $(g,h)$ in any graph $\bgr$ of the class $\Bgr$ and the vertex $(g,h)$ itself.
	This means that the vertices in $U_{gh}$ are those, which are not necessarily adjacent to $(g,h)$.
	%To distinguish between the vertices in the same row/column of $(g,h)$ in 
	%the tabular representation of the Cartesian product graph, we use the 
	%superscripts $r$ and $c$.
%\end{remark}

Besides the polynomials encountered so far, there is also a new type of 
polynomial in the Gröbner basis.
To express these, we make use of the Iverson notation.
	In particular, for a statement $A$ the value of the expression 
	$\iverson{A}$ is 1 if~$A$ 
	is true and 0 otherwise.
The following lemma shows that also these new polynomials are in 
$I_\viz$.
\begin{lemma}
	\label{lemma:gb-elements-in-ideal}
	For all vertices $(g,h) \in V(\Bgr)$ and for all choices of subsets $M \subseteq U_{gh}$, the polynomial
	\begin{align}
	\prod_{(g',h') \in \overline{U}_{gh}} (x_{g'h'} - 1) \prod_{(g,h')\in U^r_{gh}} \big(\iverson{(g&,h') \in M}(x_{gh'} - e_{hh'}) +  e_{hh'} - 1 \big) \ \, \times \nonumber \\
	&\times \prod_{(g',h)\in U^c_{gh}} \big(\iverson{(g',h) \in M}(x_{g'h} - e_{gg'} ) +  e_{gg'} - 1 \big) \label{eq:gb-long-poly}
	\end{align}
	is in the Vizing ideal $I_\viz$.	
\end{lemma}

\begin{proof}
	Let $z^* \in \Var(I_\viz)$.
	By Theorem~\ref{thm:bij-var-iviz}, $z^*$ is in bijection to a triple $(G, H, D)$ with $G \in \G$, $H \in \H$ and $D$ is a dominating set of $\bgr$.
	Assume that $z^*$ is not a zero of \eqref{eq:gb-long-poly} for some vertex~$(g,h) \in V(\Bgr)$ and some $M \subseteq U_{gh}$.
	
	Since all edge and vertex variables are boolean, this implies that all variables corresponding to vertices in $\overline{U}_{gh}$, especially $x^*_{gh}$, have to be zero. For all other vertices in $T_{gh}$ we have that the vertex variable has to be zero if the vertex is in the set~$M$ and the edge variable indicating whether there is an edge between the vertex and $(g,h)$  has to be zero if the vertex is not in $M$.
	This implies that
	\begin{equation*}
	(1 - x_{gh}^*) \Bigg( \prod\limits_{\substack{g'\in V(\G) \\ g' \neq g }} (1 - e_{gg'}^*x_{g'h}^* )  \Bigg)
	\Bigg( \prod\limits_{\substack{h'\in V(\H) \\ h' \neq h }} (1 - e_{hh'}^*x_{gh'}^*)  \Bigg) = 1
	\end{equation*}
	and therefore, $z^*$ is not a zero of~\eqref{subeqn:dom-set}, thus $z^*$ can not be a common zero of the polynomials generating $I_\viz$, which contradicts $z^* \in \Var(I_\viz)$.
	Hence, the polynomial~\eqref{eq:gb-long-poly} vanishes on $\Var(I_\viz)$ 
	and with Hilbert's Nullstellensatz  \revision{for radical ideals} 
	(Theorem~\ref{thm:hilbert}) the claim follows, \revision{as $I_\viz$ is 
		radical (Lemma~\ref{lem:Ivizradicalandfinite})}.
\end{proof}

The next two lemmas will be the main ingredients to prove that the polynomials 
generating~$I_\viz$ can be generated by the polynomials of the prospective 
Gröbner basis.
\begin{lemma}
	\label{lemma:gb-step1}
	For all vertices $(g,h) \in V(\Bgr)$, the polynomial
	\begin{align}
	\label{eq:4.10a}
	\prod_{(g,h')\in U_{gh}^r} (e_{hh'}x_{gh'} - 1) \prod_{(g',h)\in U_{gh}^c} (e_{gg'} x_{g'h} - 1) 
	\end{align}
	is equal to
	{\small\begin{align}
		\label{eq:4.10b}
		\sum_{m = 0}^{\lvert U_{gh} \rvert} \ \sum_{\substack{M \subseteq U_{gh} \\ \lvert M \rvert = m}} \ \prod_{(g',h') \in M} (x_{g'h'} - 1)\prod_{(g,h')\in U^r_{gh}} \big(e_{hh'} - \iverson{(g,h') \notin M}\big) \prod_{(g',h)\in U^c_{gh}} \big(e_{gg'} - \iverson{(g',h) \notin M}\big).
		\end{align}}%
\end{lemma}

\begin{proof}
	Using the fact that
	\begin{equation}
	\label{eq:factor-out-product-1}
	\prod_{i=1}^n (y_i - 1) = \sum_{m = 0}^n \  \sum_{\substack{M \subseteq \{1,\dots, n\} \\ \lvert M \rvert = m }} (-1)^{n-m} \prod_{i \in M} y_i
	\end{equation}
	holds for any variables $y_1, \dots, y_n$ by expanding the product, we get that~\eqref{eq:4.10a} is equal to
	\begin{equation*}
	\sum_{m = 0}^{\lvert U_{gh} \rvert} \ \sum_{\substack{M \subseteq U_{gh} \\ 
	\lvert M \rvert = m}} 
	%\bigg( 
	(-1)^{\lvert U_{gh}  \rvert - m} \prod_{(g,h')\in U^r_{gh} \cap M} 
	e_{hh'}x_{gh'} \prod_{(g',h)\in U^c_{gh} \cap M} e_{gg'}x_{g'h}%
	%\bigg)
	.
	\end{equation*}
	%, as the vertex variables in the expression on the left hand side are 
	%exactly those corresponding to the vertices in $U_{gh}$.
	
	%To finish the proof, 
	Then, 
	we consider one summand 
	of~\eqref{eq:4.10b} for a fixed~$M\subseteq U_{gh}$, so
	\begin{equation}
	\label{eq:lemma-gbstp1-product1}
	\prod_{(g',h') \in M}(x_{g'h'} - 1) \prod_{(g,h')\in U^r_{gh}} \big(e_{hh'} 
	- \iverson{(g,h') \notin M}\big) \prod_{(g',h)\in U^c_{gh}} \big(e_{gg'} - 
	\iverson{(g',h) \notin M}\big).
	\end{equation}
%	equals
%	\begin{equation*}
%	(-1)^{\lvert U_{gh}  \rvert - m} \prod_{(g,h')\in U^r_{gh} \cap M} 
%e_{hh'}x_{gh'} \prod_{(g',h)\in U^c_{gh} \cap M} e_{gg'}x_{g'h} + p_M
%	\end{equation*}
%	for all choices of $M \subseteq U_{gh}$ and $m = \lvert M \rvert$, where 
%$p_M$ is some polynomial depending on the set~$M$.
%	In a second step, we show that
%	\begin{equation*}
%	\sum_{m = 0}^{\lvert U_{gh} \rvert} \ \sum_{\substack{M \subseteq U_{gh} \\ 
%\lvert M \rvert = m}} p_M = 0
%	\end{equation*}
%	holds, which completes the proof.
%
%	Now, we start with the first step.
	Applying~\eqref{eq:factor-out-product-1} to the first product in~\eqref{eq:lemma-gbstp1-product1} yields that~\eqref{eq:lemma-gbstp1-product1}
	%	\begin{equation*}
	%		\prod_{g'h' \in M}(x_{g'h'} - 1) \prod_{(g,h')\in U^r_{gh}} \big(e_{hh'} - \iverson{(g,h') \notin M}\big) \prod_{(g',h)\in U^c_{gh}} \big(e_{gg'} - \iverson{(g',h) \notin M}\big)
	%	\end{equation*}
	equals
	\begin{align*}
	\Bigg(\prod_{(g',h') \in M}x_{g'h'} &+  \sum_{k = 0}^{m - 1} \  
	\sum_{\substack{K \subseteq M \\ \lvert K \rvert = k}} (-1)^{m-k} 
	\prod_{(g',h') \in K} x_{g'h'}  \Bigg) \ \times \\
	&\times \prod_{(g,h')\in U^r_{gh}} \big(e_{hh'} - \iverson{(g,h') \notin M}\big) \prod_{(g',h)\in U^c_{gh}} \big(e_{gg'} - \iverson{(g',h) \notin M}\big).
	\end{align*}	
	Since all vertices in $M$ are either in $U_{gh}^r$ or in $U_{gh}^c$, we can 
	rewrite this polynomial as
	%	\begin{equation*}
	%		\prod_{(g,h')\in U^r_{gh} \cap M} e_{hh'}  \prod_{(g',h)\in U^c_{gh}\cap M} e_{gg'} \prod_{(g,h')\in U^r_{gh} \setminus M} (e_{hh'} - 1) \prod_{(g',h)\in U^c_{gh}\setminus M} (e_{gg'} - 1),
	%	\end{equation*}
	{\small\begin{align}
		\label{eq:polynomialBeforepM}
		%	&\prod_{g'h' \in M}(x_{g'h'} - 1) \prod_{(g,h')\in U^r_{gh}} \big(e_{hh'} - \iverson{(g,h') \notin M}\big) \prod_{(g',h)\in U^c_{gh}} \big(e_{gg'} - \iverson{(g',h) \notin M}\big) =\\
		&\prod_{(g,h')\in U^r_{gh} \cap M} e_{hh'}x_{gh'} \prod_{(g',h)\in U^c_{gh} \cap M} e_{gg'}x_{g'h} \prod_{(g,h')\in U^r_{gh}\setminus M} (e_{hh'} - 1) \prod_{(g',h)\in U^c_{gh} \setminus M} (e_{gg'} - 1) + \\
		&\qquad +\Bigg( \sum_{k = 0}^{m-1} \sum_{\substack{K \subseteq M 
		\nonumber \\ 
		\lvert K 
		\rvert = k}} (-1)^{m-k} \prod_{(g',h') \in K} x_{g'h'}  \Bigg) \times \\
		& \qquad \qquad \times \prod_{(g,h')\in U^r_{gh}} \big(e_{hh'} - 
		\iverson{(g,h') \notin 
		M}\big) \prod_{(g',h)\in U^c_{gh}} \big(e_{gg'} - \iverson{(g',h) 
		\notin M}\big), \nonumber
		\end{align}}%
	which, using the fact that $m = \lvert M \rvert$, can be further rewritten 
	as
	\begin{equation}
	\label{eq:prodMinusOneProdProd}
	(-1)^{\lvert U_{gh} \rvert - m} \prod_{(g,h')\in U^r_{gh} \cap M} 
	e_{hh'}x_{gh'} \prod_{(g',h)\in U^c_{gh} \cap M} e_{gg'}x_{g'h} + p_M
	\end{equation}
	for some polynomial $p_M$ that depends on the set $M$ 
	and that captures the whole second summand of~\eqref{eq:polynomialBeforepM} 
	and every part of the first summand  of~\eqref{eq:polynomialBeforepM} that 
	does not contain only $-1$ in the third and fourth factor after expanding 
	the third and the fourth factor.

	%	{\small
	%	\begin{align*}
	%	&\prod_{g'h' \in M} x_{g'h'} \prod_{(g,h')\in U^r_{gh} \cap M} e_{hh'}  \prod_{(g',h)\in U^c_{gh}\cap M} e_{gg'} \prod_{(g,h')\in U^r_{gh} \setminus M} (e_{hh'} - 1) \prod_{(g',h)\in U^c_{gh}\setminus M} (e_{gg'} - 1) + \\
	%	&+ \Bigg( \sum_{k = 0}^{\lvert M - 1 \rvert} \ \sum_{\substack{K \subseteq M \\ \lvert K \rvert = k}} (-1)^{m-1-k} \prod_{x_{g'h'} \in K} x_{g'h'}  \Bigg)  \\
	%	&= \prod_{(g,h')\in U^r_{gh} \cap M} e_{hh'} x_{gh'}  \prod_{(g',h)\in U^c_{gh}\cap M} e_{gg'} x_{hg'} \prod_{(g,h')\in U^r_{gh} \setminus M} (e_{hh'} - 1) \prod_{(g',h)\in U^c_{gh}\setminus M} (e_{gg'} - 1) +  ...
	%	\end{align*}}

	To finish the proof, it remains to show %the second step,  i.e.\ 
	that the polynomial
	\begin{equation*}
	p = \sum_{m = 0}^{\lvert U_{gh} \rvert}\  \sum_{\substack{M \subseteq U_{gh} \\ \lvert M \rvert = m}} p_M
	\end{equation*}
	is equal to the zero polynomial.	
	
	All monomials in the expanded expression of $p$ 
	have in common that the number of occurring edge variables is greater than 
	the number of occurring vertex variables. 
	Indeed, 
	when expanding the product~\eqref{eq:lemma-gbstp1-product1}
%	%	\begin{equation*}
%	%	\label{prod:gb11}
%	%	\prod_{g'h' \in M}(x_{g'h'} - 1) \prod_{(g,h')\in U^r_{gh}} 
%%\big(e_{hh'} - \iverson{(g,h') \notin M}\big) \prod_{(g',h)\in U^c_{gh}} 
%%\big(e_{gg'} - \iverson{(g',h) \notin M}\big) 
%	%	\end{equation*}
	we get that 	
%	each monomial has at least $\lvert M \rvert = m$ edge variables as factor.
%	Moreover, 	
	if a vertex variable is a factor of a monomial, %of a monomial in~$p_M$, 
	the corresponding edge variable is a factor of this monomial too.
	Moreover, all monomials that have the same number of edge and vertex 
	variables are captured within the first summand of~\eqref{eq:prodMinusOneProdProd}.	
%	, it is the monomial
%	\[\prod_{(g,h')\in U^r_{gh} \cap M} e_{hh'}x_{gh'} \prod_{(g',h)\in 
%U^c_{gh} \cap M} e_{gg'}x_{g'h} \]
%	and therefore not in $p_M$ and not in $p$.
	As a result, all monomials in~$p$ have less vertex variables than edge 
	variables.
	
	\newcommand{\monomial}{q}
	
	Now, choose some fixed monomial $\monomial$
	in $p$ of degree $k+ \ell$ that 
	is a 
	product of $k$ vertex and~$\ell$ edge variables, so  
	$0 \leq k < \ell \leq \lvert U_{gh} \rvert$ holds.
	%For $0 \leq k < \ell \leq \lvert U_{gh} \rvert$ and $k < m$,
	To determine the coefficient of $\monomial$ in~$p$ 
	we count the number of different choices of the set~$M$ such that 
	$\monomial$ is a summand in~$p_M$. 
	Clearly, if $|M| = m$, then $m\geq k$ has to 
	hold. 
	The~$k$ vertex variables in $\monomial$ determine~$k$ vertices that have 
	to be in~$M$.
	Note that the corresponding edge variables are in $\monomial$ too.
	Then, there are~$\ell-k$ edge variables in $\monomial$ left that do not 
	correspond to a vertex variable.
	Of these~$\ell-k$ edge variables, $m-k$ correspond to a vertex in~$M$. 
	Therefore, there are $\binom{\ell-k}{m-k}$ different choices of~$M 
	\subseteq U_{gh}$ with~$\lvert M \rvert = m$ such that $\monomial$ is a 
	summand of $p_M$ with coefficient $(-1)^{m-k+\lvert U_{gh}\rvert - \ell}$.
	%	
	%	This is the case since out of the $\ell- k$ edges, we have to choose $m - k$ edges that correspond to a vertex in $M$, but the vertex variable does not divide the monomial. 
	%	The coefficient of this monomial in $p_M$ is $(-1)^{m + \lvert U_{gh} \rvert- (k + \ell)}$. %We multiply k + l times not by minus one
	Hence, the coefficient of $\monomial$ in $p$  equals
	\begin{equation*}
	%\label{eq:coeff-monomial-in-p}
	\sum_{m = k}^\ell \binom{\ell - k}{m-k} (-1)^{m - k + \lvert U_{gh} \rvert  -  \ell}.
	\end{equation*}
	% binomial coeff comes from edge variables
	Substituting $i = m-k$ and $n = \ell-k$, this 
	coefficient %~\eqref{eq:coeff-monomial-in-p} 
	can be written as
	\begin{equation*}
	\sum_{i= 0}^n \binom{n}{i} (-1)^{i + \lvert U_{gh} \rvert - \ell} = (-1)^{\lvert U_{gh} \rvert - \ell} \ \sum_{i= 0}^n (-1)^{i} \binom{n}{i}.
	\end{equation*}
	Using the identity
	\begin{equation*}
	\sum_{i= 0}^n (-1)^i \binom{n}{i} = 0,
	\end{equation*}
	we get that all coefficients are zero and therefore $p$ is in fact the zero polynomial, which completes the proof.
\end{proof}

\begin{lemma}
	\label{lemma:gb-step2}
	Let $(g,h) \in V(\Bgr)$, then the equations
	\begin{align*}
	%	\prod_{(g',h') \in \overline{M}(g,h)} (x_{g'h'} - 1) \cdot\prod_{(g',h)\in M_r(g,h)} (e_{gg'} x_{g'h} - 1 ) \cdot \prod_{(g,h')\in M_c(g,h)} (e_{hh'}x_{gh'} - 1)
	(e_{g_1g'}x_{g'h} - 1) p &= (x_{g'h} - 1)p +  x_{g'h} (e_{g_1g'} - 1)p \text{\quad and}\\
	(e_{h_1h'}x_{gh'} - 1) p  &=  (x_{gh'} - 1)p + x_{gh'} (e_{h_1h'} - 1)p
	\end{align*}
	hold for all $g' \in V(\G)\setminus\{g_1\}$, $h' \in V(\H)\setminus\{h_1\}$ and $p \in P_\Bgr$.
\end{lemma}

\begin{proof}
	This is straightforward to check.
\end{proof}

With all the results so far, we are now able to state the unique reduced 
Gröbner basis of~$I_\viz$ for the total degree lexicographical ordering.
\begin{theorem}
	\label{thm:gb}
	Let $k_\G = k_\H = 1$, then the reduced Gröbner basis of $I_\viz$ with respect to a total degree lexicographical term ordering consists of the  polynomials
	\begin{subequations}
		\begin{align}
		e_{g_1g} - 1 && \text{for all } g \in V(\G) \setminus \{g_1\}, \label{subeqn:gb_first}\\
		e_{h_1h} - 1 && \text{for all } h \in V(\H) \setminus \{h_1\}, \label{subeqn:gb_second}\\
		e_{gg'} (e_{gg'} - 1) && \text{for all } g \neq g' \in V(\G) \setminus \{g_1\}, \\
		e_{hh'} (e_{hh'} - 1) && \text{for all } h \neq h' \in V(\H) \setminus \{h_1\},\\
		x_{gh} (x_{gh} - 1) && \text{for all } (g,h) \in V(\Bgr), \label{subeqn:gb5}
		\end{align}
		and
		\begin{align}
		\revision{b_{gh,M} =} \prod_{(g',h') \in \overline{U}_{gh}} (x_{g'h'} - 1) \prod_{(g,h')\in U^r_{gh}}& \big(\iverson{(g,h') \in M}(x_{gh'} - e_{hh'}) +  e_{hh'} - 1 \big) \ \, \times \nonumber \\
		&\times \prod_{(g',h)\in U^c_{gh}} \big(\iverson{(g',h) \in M}(x_{g'h} - e_{gg'} ) +  e_{gg'} - 1 \big) \label{subeqn:gb-long}
		\end{align}
		for all subsets $M \subseteq U_{gh}$ for all choices of $(g,h) \in V(\Bgr)$.
	\end{subequations}
	
\end{theorem}

%\begin{remark}
	Before we start with the proof, we want to give combinatorial interpretation 
	to~\eqref{subeqn:gb-long}.
	Let $(g,h)$ be a fixed vertex in the Cartesian product graph $\bgr$.
	As already mentioned, the vertices in $\overline{U}_{gh}$ are $(g,h)$ and all vertices that are adjacent to $(g,h)$ in all product graphs~$\bgr$ of the graph class $\Bgr$.
	
	Let $D$ be a dominating set (of any size) in~$\bgr$. If there is a vertex in~$\overline{U}_{gh} \cap D$, the vertex~$(g,h)$ is dominated by this vertex in~$D$.
	If this is not the case, then there has to be a vertex in~$U_{gh} \cap D$, that is adjacent to~$(g,h)$.
	%For the choice of~$M = U_{gh}$, the polynomial~\eqref{subeqn:gb-long}
	\revision{The polynomial $b_{gh,U_{gh}}$} ensures that at least one vertex in $U_{gh}$ is in $D$. The above choice of $M$ does not assure that a vertex adjacent to $(g,h)$ in $U_{gh}$ is in $D$. Indeed, assume that there is no vertex in~$D \cap U_{gh}$ that is adjacent to $(g,h)$. 
	\revision{Then all vertex variables occurring in $b_{gh, U_{gh}\setminus D}$ are zero because the corresponding vertices are not in~$D$.}
	%Choose $M = U_{gh}\setminus D$, then all vertex variables that occur in the polynomial~\eqref{subeqn:gb-long} are zero because the corresponding vertices are not in~$D$.
	Hence, there has to be at least one edge between $(g,h)$ and a vertex in~$D \cap U_{gh}$. This ensures that the vertex $(g,h)$ is dominated by $D$.
	
	To sum this up, the polynomials in~\eqref{subeqn:gb-long} in the Gröbner 
	basis guarantee that $(g,h)$ is dominated by a vertex in $D$.
	Next, we present the proof of Theorem~\ref{thm:gb}.
%\end{remark}

\begin{proof}[Proof of Theorem~\ref{thm:gb}]
	Let $B$ be the set of all polynomials in the claimed reduced Gröbner basis.
	First, we will show that the polynomials in $B$ are indeed in $I_\viz$.
	Then, we will show that the leading term of each polynomial $f$ in $I_\viz$ is divisible by the leading term of some polynomial in $B$.
	The third step in the proof will be to show that the polynomials in $B$ are a generating system of $I_\viz$, hence after this step we know that~$B$ is a Gröbner basis.
	The last step will be to show that $B$ is even a reduced Gröbner basis.
	
	From~\eqref{subeqn:edges}, \eqref{subeqn:vertices} and 
	Lemma~\ref{lemma:1-e-in-ideal} and~\ref{lemma:gb-elements-in-ideal} we get 
	that all polynomials in $B$ are in~$I_\viz$, so the first step of the proof 
	is easily finished.

	For the second step, let us consider the divisibility of the leading terms.
	We show that the desired property holds for each of the polynomials we used 
	to generate $I_\viz$, which then implies the property for all polynomials 
	$f$ in $I_\viz$. 
	For the polynomials in \eqref{subeqn:edges}, \eqref{subeqn:dom-edge} and \eqref{subeqn:vertices} this is trivial.
	Since $k_\G = k_\H = 1$ in our setting, there are no polynomials in \eqref{subeqn:min-dom}.
	Furthermore, for all $(g,h) \in V(\Bgr)$ the leading term of~\eqref{subeqn:dom-set}, that is
	\begin{equation*}
	(-1)^{n_\G + n_\H - 1}	x_{gh}\Bigg( \prod\limits_{\substack{g'\in V(\G) \\ g' \neq g }}  e_{gg'}x_{g'h} \Bigg)
	\Bigg( \prod\limits_{\substack{h'\in V(\H) \\ h' \neq h }} e_{hh'}x_{gh'}  \Bigg),
	\end{equation*}
	is divisible by the leading term of \eqref{subeqn:gb-long} for $M = U_{gh}$, which equals $\rho^{n_\G + n_\H - 1}_{gh}$.
	
	As a third step, we prove that $B$ is a generating system of $I_\viz$ by representing the polynomials of Definition~\ref{def:ideal-graph} and \ref{def:ideal-prodgraph} in terms of the polynomials in $B$.
	This is again easy to check, except for the polynomials~\eqref{subeqn:dom-set}.
	For a fixed vertex $(g,h) \in V(\Bgr)$, we will build the 
	polynomial~\eqref{subeqn:dom-set}
	step by step using polynomials of $B$.
	%Due to Lemma~\ref{lemma:gb-step1}, clearly~\eqref{eq:4.10a} and~\eqref{eq:4.10b}
	First, we sum up \revision{$b_{gh,M}$} multiplied by 
	\begin{equation*}
	\prod_{(g,h')\in U^r_{gh} \cap M} e_{hh'}  \prod_{(g',h)\in U^c_{gh} \cap M} e_{gg'} \in P_\Bgr
	\end{equation*}
	\revision{for all possible subsets $M$ of $U_{gh}$.}
	\revision{Since $b_{gh,M} \in B$}, this sum can be represented by polynomials in~$B$ and 
	equals~\eqref{eq:4.10b} multiplied with \[\prod_{(g',h')\in 
	\overline{U}_{gh}}(x_{g'h'} - 1).\]
	Due to Lemma~\ref{lemma:gb-step1}, this sum is also equal to
	\begin{equation}
	\label{eq:gb-proof-before-last-step}
	\prod_{(g',h')\in \overline{U}_{gh}}(x_{g'h'} - 1) \prod_{(g,h')\in U_{gh}^r} (e_{hh'}x_{gh'} - 1) \prod_{(g',h) \in U_{gh}^c} (e_{gg'} x_{g'h} - 1).
	\end{equation}

	%	By multiplying the equation in Lemma~\ref{lemma:gb-step1} with \[\prod_{(g',h')\in \overline{U}_{gh}}(x_{g'h'} - 1),\] we get the polynomial
	%	\begin{equation*}
	%	\prod_{(g',h')\in \overline{U}_{gh}}(x_{g'h'} - 1) \prod_{\substack{h' \in V(\H) \setminus \{h_1,h\}\\ h \neq h_1}} (e_{hh'}x_{gh'} - 1) \prod_{\substack{g' \in V(\G)\setminus \{g_1,g\} \\ g \neq g_1 }} (e_{gg'} x_{g'h} - 1),
	%	\end{equation*}
	%	which is the sum of all polynomials in \eqref{subeqn:gb-long} multiplied by 
	%	\begin{equation*}
	%	\prod_{(g,h')\in U^r_{gh} \cap M} e_{hh'}  \prod_{(g',h)\in U^c_{gh} \cap M} e_{gg'} \in P_\Bgr.
	%	\end{equation*}
	
	Next, we iteratively apply Lemma~\ref{lemma:gb-step2} for all vertices in $\overline{U}_{gh}\setminus \{(g,h)\}$ to obtain from~\eqref{eq:gb-proof-before-last-step} the polynomial
	\begin{equation*}
	(x_{gh} - 1) \Bigg( \prod\limits_{\substack{g'\in V(\G) \\ g' \neq g }} (e_{gg'}x_{g'h} - 1)  \Bigg)
	\Bigg( \prod\limits_{\substack{h'\in V(\H) \\ h' \neq h }} (e_{hh'}x_{gh'} 
	- 1)  \Bigg)
	\end{equation*}
	in the following way.
	First we fix a vertex $(g',h') \in \overline{U}_{gh}\setminus \{(g,h)\}$.
	Let the polynomial~$p$ be such that~\eqref{eq:gb-proof-before-last-step} equals $(x_{g'h'} - 1) p$.
	By $e$ we denote the edge variable corresponding to $(g',h')$, that is~$e_{g_1g'}$ if $h'=h$ and $e_{h_1h'}$ otherwise.
	Now, we add to $(x_{g'h'}-1)p$ the polynomial $(e-1)x_{g'h'}p$ to obtain $(ex_{g'h'}-1)p$ by Lemma~\ref{lemma:gb-step2}.
	Since we added a polynomial
	from \eqref{subeqn:gb_first} or  \eqref{subeqn:gb_second} times a polynomial in $P_\Bgr$ to a polynomial generated by $B$, the resulting polynomial is again generated by $B$.
	Based on this new polynomial, we choose the next vertex in $\overline{U}_{gh}\setminus \{(g,h)\}$ and apply the same arguments as before to this polynomial.
	This is done for all vertices in $\overline{U}_{gh}\setminus \{(g,h)\}$.
	%	In each iteration of the last step of the construction we added a polynomial from \eqref{subeqn:gb_first} or  \eqref{subeqn:gb_second} times a polynomial depending on the current vertex variable.
	
	Finally, multiplying by $(-1)^{n_\G + n_\H - 1}$ yields that the requested polynomial~\eqref{subeqn:dom-set} can be generated by $B$.
	This finalizes to prove that $B$ is a Gröbner basis.
	
	The last step in the proof is to show that $B$ is a reduced Gröbner basis.	
	It is rather easy to see that there is no monomial in any of the polynomials in \eqref{subeqn:gb_first}--\eqref{subeqn:gb5}, which can be represented by the leading terms of the other polynomials in $B$.
	Moreover, it holds that the leading term of any polynomial from~\eqref{subeqn:gb-long} is the product of all variables in the polynomial that do not cancel out.
	The leading terms of these polynomials are of the same degree, square-free and pairwise distinct.
	Moreover, it holds that the variables in the leading terms of~\eqref{subeqn:gb-long} are not leading term of any polynomial from~\eqref{subeqn:gb_first} or~\eqref{subeqn:gb_second} in $B$.
	Therefore, we can not represent a leading term of a polynomial from~\eqref{subeqn:gb-long} by the leading terms of the other polynomials in~$B$.
	A monomial $m_1$ of a polynomial $p$ of type~\eqref{subeqn:gb-long} in~$B$, 
	which is not the leading term, has a smaller degree than the leading term 
	and is the product of pairwise distinct variables, which occur in the 
	polynomials of degree 2 in~$B$.
	For each leading term of the polynomials of type~\eqref{subeqn:gb-long} it holds that there is a variable which is a factor of the leading term but is no factor of $m_1$.
	Due to these facts, we are not able to represent~$m_1$ by the leading terms of $B \setminus \{p\}$.
	%We are also not able to represent the leading term of $f$ by the leading terms of the polynomials in $B \setminus \{p\}$.
	%	Therefore, we are not able to represent the monomial~$m$ without the leading term of a polynomial from~\eqref{subeqn:gb-long}.
	%	But as the degree of the leading terms is larger than the degree of~$m$, we are also not able to represent~$m$ with the leading terms of all 
	%	Therefore, it holds that each monomial~$m$ of a fixed polynomial of~\eqref{subeqn:gb-long} can not be represented by the leading terms of all other polynomials in $B$.
	%	
	%	This is the case, because $m$ is a product of pairwise distinct variables, is not divisible by the variables $e_{h_1h}$ and $e_{g_1g}$, in each leading term of~\eqref{subeqn:gb-long} there is a variable that is in the leading term but not in $m$ and the other leading terms are squares.
	%	
	%	Moreover, there is no monomial in the polynomials from~\eqref{subeqn:gb-long}, that is divisible by the leading terms of \eqref{subeqn:gb_first}--\eqref{subeqn:gb5}, since the variables $e_{h_1h}$ and $e_{g_1g}$ are not used and there is no square of a variable in any monomial. %e_1h and e_g1 not used and no square of any variable in a monomial
	%	The leading term of a polynomial from~\eqref{subeqn:gb-long} is the product of all variables appearing in the polynomial. The leading terms of these polynomials are of the same degree and pairwise distinct. This implies that no monomial of one polynomial is divisible by the leading term of another polynomial because at least one variable is in this leading term but not in the monomial. 
	Together with the fact that all polynomials in $B$ have leading coefficient 
	1, we conclude that the Gröbner basis is reduced.
\end{proof}

With the help of the reduced Gröbner basis of~$I_\viz$ obtained in 
Theorem~\ref{thm:gb}, we know that if a 
polynomial is not representable in terms of the polynomials in this basis, then 
it can not be in~$I_\viz$.
We use this to get an SDP formulation to computationally find an 
SOS-certificate.
Before doing so, we use the Gröbner basis to determine the minimum degree 
of an SOS-certificate.

\subsection{Minimum Degree of a Sum-Of-Squares Certificate}
The knowledge of the reduced Gröbner basis of $I_\viz$ allows us to state a 
lower bound on the degree~$\ell$ of an SOS-certificate for Vizing's conjecture 
in the case of $k_\G = k_\H = 1$.
\begin{theorem}
	\label{thm:min-degree-cert}
	Let $k_\G = k_\H = 1$ and $n_\G$, $n_\H > 1$, then there is no 
	$\ell$-SOS-certificate of $f_\viz$ for any integer $\ell$ less than $(n_\G 
	+ n_\H - 1)/2$.
\end{theorem}
\begin{proof}
	For any set of polynomials $s_1, \dots, s_t \in P_\Bgr$ that forms an 
	$\ell$-SOS-certificate of $f_\viz$, it needs to hold that
	\begin{equation*}
	\sum_{i=1}^{t} s_i^2 - f_\viz \equiv 0 \quad \mod I_\viz.
	\end{equation*}
	Additionally, the degrees of the polynomials~$s_1,\dots s_t$ have to be at most~$\ell$.
	
	For all $1 \leq i \leq t$ let $p_i$  be the polynomial that results from 
	$s_i$ by evaluating $e_{g_1g} = 1$ and~$e_{h_1h} = 1$ for all $g \in 
	V(\G)\setminus \{g_1\}$ and for all $h \in V(\H) \setminus \{h_1\}$.
	Lemma~\ref{lemma:1-e-in-ideal} yields that
	\begin{equation*}
	\sum_{i=1}^{t} p_i^2 - f_\viz \equiv \sum_{i=1}^{t} s_i^2 - f_\viz \equiv 0 \quad \mod I_\viz.
	\end{equation*}
	
	To show that something is congruent to $0$ modulo~$I_\viz$ is the same as proving that it is contained in~$I_\viz$.
	This implies that
	\begin{equation}
	\label{eq:proof-mindegree-poly}
	\sum_{i=1}^{t} p_i^2 - f_\viz = \sum_{i=1}^{t} p_i^2 - \sum_{(g,h) \in V(\Bgr)} x_{gh} + 1
	\end{equation}
	has to be generated by the elements in the Gröbner basis of~$I_\viz$ stated 
	in Theorem~\ref{thm:gb}.
	
	Assume that this can be done by using the elements of degree $1$ and $2$ only.
	We know that the constant term of the polynomial 
	in~\eqref{eq:proof-mindegree-poly} is greater or equal to 1.
	Furthermore, we are not able to represent a polynomial with constant term other than zero by using only the Gröbner basis elements of degree~2.
	This means that we have to use at least one Gröbner basis element of degree 
	1.
	But if we do so, we end up getting an edge variable $e_{g_1g}$ with $g \in 
	V(\G)\setminus \{g_1\}$ or~$e_{h_1h}$ with $h \in V(\H) \setminus\{h_1\}$ 
	in the resulting polynomial.
	Clearly, there is no such edge variable in~\eqref{eq:proof-mindegree-poly}.
	Therefore, it holds that it is not possible to represent the polynomial in~\eqref{eq:proof-mindegree-poly} by using the elements of degree $1$ and $2$ only.
	
	Intuitively, this makes sense, as the polynomials in the Gröbner basis of degree 1 can be used to reduce the variable to $1$, and the polynomials of degree 2 can be used to reduce higher powers of the variable to the variable itself, and this is not enough to reduce~\eqref{eq:proof-mindegree-poly} to zero.
	
	However, all other polynomials in the Gröbner basis have degree~$n_\G + n_H - 1$.
	Hence, the degree of the polynomial in~\eqref{eq:proof-mindegree-poly} is at least~$n_\G + n_H - 1$.
	Consequently, there has to be at least one polynomial $p_i$ such that the degree of $p_i^2$ is greater or equal to~$n_\G + n_\H - 1$.
	As the degree of~$p_i$ is less or equal to the degree of $s_i$, we get that two times the degree of $s_i$ is also greater or equal to $n_\G + n_\H - 1$.
	This implies that there is no $\ell$-SOS-certificate of~$f_\viz$ for~$\ell 
	< (n_\G + n_\H - 1)/2$.
\end{proof}

As a result of Theorem~\ref{thm:min-degree-cert}, any $\ell$-SOS-certificate 
for  $k_\G = k_\H = 1$ has to be at least of degree 
$\ell \geq (n_\G + n_\H - 1)/2$. This is the first result stating the minimum 
degree of an $\ell$-SOS-certificate for any values of $n_\G$, $n_\H$, $k_\G$ 
and $k_\H$.

\section{New Certificates for Two Subclasses of 
	\texorpdfstring{$k_\G = k_\H = 1$}{kG=kH=1}}
\label{chapter:examples}

In this section, we present SOS-certificates for Vizing's conjecture
obtained with the method of~\cite{vizing-long2020}, i.e., 
by following the steps in~\cite[Section 4]{vizing-long2020}, 
and recalled in Section~\ref{chapter:thbackground}.
To set up the SDP, to solve the SDP and to computationally verify our obtained 
certificates, we made use of the code provided in~\cite{vizing-long2020}.
In particular, we ran the code in SageMath~\cite{sagemath} and in MATLAB using 
MOSEK~\cite{mosek}.
We refrain from detailing the steps and focus on presenting the 
SOS-certificates and proving their correctness. For details on how they were 
obtained we refer to the master thesis of 
Siebenhofer~\cite{melanieMasterThesisVizing}.

\subsection{Certificate for 
	\texorpdfstring{$n_\G = 3$}{nG=3},
	\texorpdfstring{$n_\H = 2$}{nH=2} and 
	\texorpdfstring{$k_\G = k_\H = 1$}{kG=kH=1}}

We start by presenting a  $2$-SOS-certificate
for the case $n_\G = 3$, $n_\H = 2$ and $k_\G = k_\H = 1$\revision{.}

\begin{theorem}
	\label{thm:cert32}
	For $n_\G = 3$, $n_\H = 2$ and $k_\G = k_\H = 1$, Vizing's conjecture 
	holds, as for all choices of $(g,h) \in V(\Bgr) $ the polynomials
	\begin{align*}
	s_{g^*h^*} &= x_{g^*h^*}  \text{\hspace*{2.5cm} for all $(g^*,h^*) \in 
	V(\Bgr)$  with $g^*\neq g$ and $ h^*\neq h$,}\\
	s_1 &= -\alpha + \alpha \sum_{(g',h') \in T_{gh}} x_{g'h'} + \beta 
	\sum_{(g',h') \in T_{gh}} x_{g'h'}  \sum_{(g'',h'') \in 
	T_{gh}\setminus{\{(g',h')\}} } x_{g''h''} \quad  \text{ and}\\
	s_2 &= \delta \sum_{(g',h') \in T_{gh}} x_{g'h'}  \sum_{(g'',h'') \in T_{gh}\setminus{\{(g',h')\}} } x_{g''h''},
	\end{align*}
	form a 2-SOS-certificate of $f_\viz$ for all
	$(\alpha, \beta, \delta)$ in
	\[ \bigg\{ \Big(-\sqrt{3}, \frac{4}{9}\sqrt{3}, -\frac{1}{9} \sqrt{6} \Big),
	\Big(-\sqrt{3}, \frac{4}{9}\sqrt{3}, \frac{1}{9} \sqrt{6} \Big),
	\Big(\sqrt{3}, -\frac{4}{9}\sqrt{3}, -\frac{1}{9} \sqrt{6}\Big),
	\Big(\sqrt{3}, -\frac{4}{9}\sqrt{3}, \frac{1}{9} \sqrt{6}\Big)
	\bigg\}. \]
\end{theorem}	
\begin{remark}
	\label{rmk:cert32}
	Theorem~\ref{thm:cert32} is true whenever
	$\alpha$, $\beta$, $\delta$ are solutions to the system of equations
	\begin{align*}
	\alpha^2 + 1 &= 2\alpha^2 + \beta^2 + 2\alpha\beta + \delta^2, \\
	-(\alpha^2 + 1) &= 6\beta^2 + 6\alpha\beta + 6\delta^2 \quad \text{and} \\
	\alpha^2 + 1 &=6\beta^2 + 6\delta^2.
	\end{align*}
	It can be checked that the pairs~$(\alpha, \beta, \delta)$ stated in Theorem~\ref{thm:cert32} are all solutions to this system of equations.
\end{remark}	
\begin{proof}[Proof of Theorem~\ref{thm:cert32}]
	We start by fixing a vertex~$(g,h) \in V(\Bgr)$.
	Next, we write the polynomials~$s_1$ and~$s_2$ using the polynomials $\rho_{gh}^1$ and $\rho_{gh}^2$ from Definition~\ref{def:rho}, so
	\begin{align*}
	s_1 &= -\alpha + \alpha \rho_{gh}^1 + \beta \rho_{gh}^2 \quad \text{ and}\\
	s_2 &= \delta \rho_{gh}^2.
	\end{align*} 
	For brevity, we denote by $\rho^k$ the polynomial $\rho^k_{gh}$ for $1 \leq k \leq 4$.
	By Corollary~\ref{cor:reduce-prod-polynomials} we get
	\begin{align*}
	\rho^1\rho^1 & \equiv 
	%\binom{1}{0}\binom{1}{1} \rho^1 + \binom{1}{1} \binom{2}{1} \rho^2 
	%= 
	\rho^1 + 2\rho^2 \quad \mod I_\viz ,\\
	\rho^2\rho^2 & \equiv 
	%\binom{2}{0} \binom{2}{2} \rho^2 + \binom{2}{1}\binom{3}{2} \rho^3 + 
	%\binom{2}{2}\binom{4}{2} \rho^4 
	%= 
	\rho^2 + 6\rho^3 + 6 \rho^4 \quad \mod I_\viz  \quad \text{and}\\
	\rho^1\rho^2& \equiv 
	%\binom{1}{0} \binom{2}{1} \rho^2 + \binom{1}{1}\binom{3}{1} \rho^3 
	%= 
	2\rho^2 + 3\rho^3 \quad \mod I_\viz.
	\end{align*}
	These congruences imply that
	\begin{align*}
	s_1^2 & = (-\alpha + \alpha \rho^1 + \beta \rho^2)^2\\
	&= \alpha^2 + \alpha^2 \rho^1 \rho^1 + \beta^2 \rho^2 \rho^2 - 2\alpha^2 \rho^1 - 2 \alpha \beta \rho^2 + 2 \alpha \beta \rho^1 \rho^2\\
	&\equiv \alpha^2  + \alpha^2(\rho^1 + 2\rho^2) + \beta^2(\rho^2 + 6\rho^3 + 6\rho^4) - 2\alpha^2\rho^1  - 2\alpha\beta \rho^2 + 2\alpha\beta(2\rho^2 + 3\rho^3) = \\
	& = \alpha^2 - \alpha^2 \rho^1 + (2\alpha^2 + \beta^2 + 2\alpha\beta)\rho^2 + (6\beta^2 + 6\alpha\beta)\rho^3 + 6\beta^2 \rho^4 \quad \mod I_\viz
	\end{align*}
	and
	\begin{align*}
	s_2^2 & = (\delta \rho^2)^2 \equiv  \delta^2\rho^2 + 6\delta^2 \rho^3 + 6\delta^2 \rho^4 \quad \mod I_\viz 
	\end{align*}
	hold.
	
	The sum of squares of the polynomials in the certificate can be written as
	\begin{multline}
	\label{eq:sos-cert-32}
	\sum_{\substack{(g^*,h^*) \in V(\Bgr) \\ g^* \neq g, h^* \neq h}} 
	s_{g^*h^*}^2 + s_1^2 + s_2^2  = \sum_{\substack{(g^*,h^*) \in V(\Bgr) \\ 
	g^* \neq g, h^* \neq h}} x_{g^*h^*}^2 + \alpha^2  - \alpha^2 \rho^1  + 
	(2\alpha^2 + \beta^2 + 2\alpha\beta + \delta^2)\rho^2 \\
	+ (6\beta^2 + 6\alpha\beta + 6\delta^2)\rho^3 + (6\beta^2 + 6\delta^2)\rho^4.
	\end{multline}
	Using the fact that $x_{g^*h^*}^2 \equiv x_{g^*h^*}  \ \mod I_\viz $, we get
	\begin{align*}
	\sum_{\substack{(g^*,h^*) \in V(\Bgr) \\ g^* \neq g, h^* \neq h}} x_{g^*h^*}^2 + \rho^1 - 1 &\equiv \sum_{\substack{(g^*,h^*) \in V(\Bgr) \\ g^* \neq g, h^* \neq h}} x_{g^*h^*} + \rho^1 - 1 \\
	& = \sum_{(g^*,h^*) \in V(\Bgr) \setminus T_{gh}} x_{g^*h^*} + \sum_{(g',h') \in T_{gh}} x_{g'h'} - 1\\
	& = f_\viz \quad \mod I_\viz.
	\end{align*}
	Therefore the sum of squares~\eqref{eq:sos-cert-32} written as
	\begin{multline*}
	\sum_{\substack{(g^*,h^*) \in V(\Bgr) \\ g^* \neq g, h^* \neq h}} x_{g^*h^*}^2 + \rho^1 - 1 + (\alpha^2 + 1) - (\alpha^2 + 1) \rho^1\\ + (2\alpha^2 + \beta^2 + 2\alpha\beta + \delta^2)\rho^2 + (6\beta^2 + 6\alpha\beta + 6\delta^2)\rho^3 + (6\beta^2 + 6\delta^2)\rho^4
	\end{multline*}
	is congruent
	\begin{multline*}
	f_\viz + (\alpha^2 + 1) - (\alpha^2 + 1) \rho^1 + (2\alpha^2 + \beta^2 + 2\alpha\beta + \delta^2)\rho^2 + (6\beta^2 + 6\alpha\beta + 6\delta^2)\rho^3 + (6\beta^2 + 6\delta^2)\rho^4
	\end{multline*}
	modulo $I_\viz$.
	Since~$\alpha$, $\beta$ and $\delta$ satisfy
	\begin{align*}
	\alpha^2 + 1 &= 2\alpha^2 + \beta^2 + 2\alpha\beta + \delta^2, \\
	-(\alpha^2 + 1) &= 6\beta^2 + 6\alpha\beta + 6\delta^2 \quad \text{and} \\
	\alpha^2 + 1 &=6\beta^2 + 6\delta^2,
	\end{align*}
	the sum of squares of the polynomials is congruent
	\begin{align*}
	f_\viz + (\alpha^2 + 1)(1 - \rho^1 + \rho^2 - \rho^3 + \rho^4)
	\end{align*}
	modulo~$I_\viz$.
	Lemma~\ref{lemma:reduction-polynomial-as-sum} together with Lemma~\ref{lemma:reduction-polynomial-in-ideal} yields that
	\begin{align*}
	f_\viz + (\alpha^2 + 1)(1 - \rho^1 + \rho^2 - \rho^3 + \rho^4) &=  f_\viz + (\alpha^2 + 1) \prod_{(g',h')\in T_{gh}} (1 - x_{g'h'}) \\
	& \equiv f_\viz \quad \mod I_\viz,
	\end{align*}
	which completes the proof.
\end{proof}

To sum up, we found for each of the 6 vertices in $\Bgr$ 4 different 
certificates of degree~$2$ for Vizing's conjecture on the graph class $\Bgr$ 
with 
$n_\G = 3$, $k_\G = 1$, $n_\H = 2$ and $k_\H = 1$. In total, \revision{these 
give} $24$ 
different $2$-SOS-certificates.

\subsection{Certificate for 
	\texorpdfstring{$n_\G = n_\H = 2$}{nG=nH=2} and 
	\texorpdfstring{$k_\G = k_\H = 1$}{kG=kH=1}}

Next, we consider the graph class with $n_\G = n_\H = 
2$ and 
{$k_\G = k_\H = 1$. Here we find the following certificate.

\begin{theorem}
	\label{thm:cert2121}
	For $n_\G = n_\H = 2$ and $k_\G = k_\H = 1$ Vizing's conjecture holds, 
	since for any $(g,h) \in V(\Bgr)$ the two polynomials
	\begin{align*}
	s_{g^*h^*} &= x_{g^*h^*} && \text{and}\\
	s_1 &= -\alpha + \alpha \sum_{(g',h') \in T_{gh}} x_{g'h'} + \beta \sum_{(g',h') \in T_{gh}} x_{g'h'}  \sum_{(g'',h'') \in T_{gh}\setminus{\{(g',h')\}} } x_{g''h''}
	\end{align*}
	with  $(g^*,h^*)$ being the only vertex not in the set $T_{gh}$,
	form a 2-SOS-certificate of $f_\viz$ for all pairs
	\[(\alpha, \beta) \in \Big\{ \big(\sqrt{2} + 3, - \sqrt{2} - 2\big), \big(-\sqrt{2} + 3, \sqrt{2} - 2\big), \big(  \sqrt{2} - 3, - \sqrt{2} + 2\big), \big(-\sqrt{2} - 3, \sqrt{2} + 2\big) \Big\}.\]
\end{theorem}

\begin{remark}
	\label{rmk:cert2121}
	In particular, Theorem~\ref{thm:cert2121} is true whenever $\alpha$ and $\beta \in \R$	are solutions to the system of equations
	\begin{align*}
	\alpha^2 + 1 &= 2\alpha^2 + \beta^2 + 2\alpha\beta \text{\quad and}\\
	-(\alpha^2 + 1) &= 6 \beta^2 + 6\alpha\beta.
	\end{align*}
	The ones stated in the theorem are all solutions to this system of equations.
\end{remark}

\begin{proof}[Proof of Theorem~\ref{thm:cert2121}]
	This proof is analogous to that of Theorem~\ref{thm:cert32}.
	First, we fix a vertex~$(g,h) \in V(\Bgr)$.
	Next, we rewrite $s_1$ as $- \alpha + \alpha \rho_{gh}^1 + \beta 
	\rho_{gh}^2$.
	For the sake of brevity, we denote by $\rho^k$ the polynomial $\rho_{gh}^k$ for $1 \leq k \leq 3$.
	By Corollary~\ref{cor:reduce-prod-polynomials} we get
	\begin{align*}
	\rho^1 \rho^1 &\equiv 
	%\binom{1}{0} \binom{1}{1} \rho^1 + \binom{1}{1} \binom{2}{1} \rho^2 
	%= 
	\rho^1 + 2 \rho^2 \quad \mod I_\viz,\\
	\rho^2 \rho^2 &\equiv 
	%\binom{2}{0}\binom{2}{2} \rho^2 + \binom{2}{1}\binom{3}{2} \rho^3 
	%= 
	\rho^2 + 6 \rho^3 \quad \mod I_\viz \quad \text{ and} \\
	\rho^1 \rho^2 &\equiv 
	%\binom{1}{0} \binom{2}{1} \rho^2 + \binom{1}{1} \binom{3}{1} \rho^3 
	%= 
	2 \rho^2 + 3 \rho^3 \quad \mod I_\viz.
	\end{align*}
	Hence, we can write $s_1^2$ as 
	\begin{align*}
	s_1^2 &= (-\alpha + \alpha \rho^1 + \beta \rho^2)^2 \\
	&= \alpha^2 + \alpha^2 \rho^1 \rho^1 + \beta^2 \rho^2 \rho^2  - 2 \alpha^2 \rho^1 - 2 \alpha\beta \rho^2 + 2 \alpha \beta \rho^1 \rho^2\\
	&\equiv \alpha^2  + \alpha^2(\rho^1 + 2 \rho^2) + \beta^2(\rho^2 + 6 \rho^3) - 2\alpha^2 \rho^1 - 2\alpha\beta \rho^2  + 2\alpha\beta(2\rho^2 + 3\rho^3) = \\
	& = \alpha^2 + (\alpha^2 - 2\alpha^2)\rho^1 + (2\alpha^2 + \beta^2 - 2\alpha\beta + 4\alpha\beta)\rho^2 + (6\beta^2 + 6\alpha\beta)\rho^3 = \\
	& = \alpha^2  - \alpha^2 \rho^1 + (2\alpha^2 + \beta^2 + 2\alpha\beta)\rho^2 + (6\beta^2 + 6\alpha\beta)\rho^3 \quad \mod I_\viz.
	\end{align*}
	Using the fact that $x_{g^*h^*}^2 \equiv x_{g^*h^*}  \ \mod I_\viz$ holds, we get that 
	\begin{align*}
	x_{g^*h^*}^2 + \rho^1 - 1 &\equiv x_{g^*h^*} + \rho^1 - 1 \\
	& = x_{g^*h^*} + \sum_{(g',h') \in T_{gh}} x_{g'h'} - 1 \\
	& = f_\viz \quad \mod I_\viz.
	\end{align*}
	Therefore, for the sum of the polynomials squared it holds that
	\begin{align*}
	x_{g^*h^*}^2 & + s_1^2 \equiv x_{g^*h^*}^2 + \alpha^2  - \alpha^2 \rho^1 + 
	(2\alpha^2 + \beta^2 + 2\alpha\beta)\rho^2 + (6\beta^2 + 
	6\alpha\beta)\rho^3 \\
	& = x_{g^*h^*}^2 + \rho^1 - 1  + (\alpha^2 + 1)  - (\alpha^2 + 1) \rho^1 + (2\alpha^2 + \beta^2 + 2\alpha\beta)\rho^2 + (6\beta^2 + 6\alpha\beta)\rho^3 \\
	& \equiv f_\viz + (\alpha^2 + 1)  - (\alpha^2 + 1) \rho^1 + (2\alpha^2 + \beta^2 + 2\alpha\beta)\rho^2 + (6\beta^2 + 6\alpha\beta)\rho^3 \quad \mod I_\viz.
	\end{align*}
	Since $\alpha$ and $\beta$ satisfy
	\begin{align*}
	\alpha^2 + 1 &= 2\alpha^2 + \beta^2 + 2\alpha\beta \text{\quad and}\\
	-(\alpha^2 + 1) &= 6 \beta^2 + 6\alpha\beta,
	\end{align*}
	we can further conclude with Lemma~\ref{lemma:reduction-polynomial-as-sum} and Lemma~\ref{lemma:reduction-polynomial-in-ideal} that
	\begin{align*}
	x_{g^*h^*}^2 + s_1^2 &\equiv f_\viz + (\alpha^2 + 1)  - (\alpha^2 + 1) \rho^1 + (2\alpha^2 + \beta^2 + 2\alpha\beta)\rho^2 + (6\beta^2 + 6\alpha\beta)\rho^3  \\
	& = f_\viz + (\alpha^2 + 1)(1 - \rho^1 + \rho^2 - \rho^3) \\
	& \equiv f_\viz \quad \mod I_\viz
	\end{align*}
	holds, which closes the proof.
\end{proof}

One may observe the strong parallelism between the two graph classes 
considered. %$n_\G = n_\H = 2$ and $n_G = 3$ and $n_\H = 2$.
In the next section we derive a generalized method to find certificates of this 
special form.

\section{General Approach to Find Certificates 
	for~\texorpdfstring{$k_\G = k_\H = 1$}{kG=kH=1}}
\label{chapter:generalapproach}

In this section, we first give a general formulation of the previous two 
SOS-certificates, that could potentially be an  SOS-certificate for any graph 
classes~$\G$ and~$\H$ with~$k_\G = k_\H = 1$.
To really obtain an SOS-certificate one has to determine the coefficients of 
the polynomials in this specific SOS-certificate by finding a solution of a 
system of equations.
We give an algorithm to find such a solution in the second part of this section.

\subsection{General Certificate 
	for \texorpdfstring{$k_\G = k_\H = 1$}{kG=kH=1}}
The SOS-certificates of the last section have a few things in common.
First, a vertex~$(g,h) \in V(\Bgr)$ is selected, which determines the set~$T_{gh}$ as defined at the beginning of Section~\ref{sec:auxiliary-results}.
Then, the polynomials of degree greater than~1 in the certificate contain only vertex variables corresponding to the vertices in~$T_{gh}$.
In particular, we use the polynomials~$\rho_{gh}^i$ from Definition~\ref{def:rho} to represent these polynomials in the certificate.
Based on computational results and on our knowledge of the Gröbner basis, we 
propose the following specific form of a possible SOS-certificate.
Additionally, we give a condition on the correctness of the certificate in the next theorem.
\begin{theorem}
	\label{thm:general-cert}
	Let~$k_\G = k_\H = 1$ and let $d = n_\G + n_\H - 1$.
	If $c_{w,i} \in \R$ for~$1 \leq w \leq \lceil d/2 \rceil$ and~$0 \leq i \leq 
	\lceil 
	d/2 \rceil$ is 
	a solution to the system of equations
%	\begin{subequations}
%	\begin{align}
%	(-1) \bigg(\sum_{w=1}^{\lceil d/2 \rceil} c_{w0}^2 + 1\bigg) &=  
%	-1 + \sum_{w=1}^{\lceil 
%	d/2 \rceil} ( c_{w1}^2 + 2 c_{w1}c_{w0} )
%	\quad \text{ and }\\
%	(-1)^k \bigg(\sum_{w=1}^{\lceil d/2 \rceil} c_{w0}^2 + 1 \bigg) 
%	&= 
%	\sum_{i = \lceil k/2 \rceil}^{\min\{k, \lceil d/2 \rceil\}} \sum_{w 
%	=1}^{\lceil d/2 \rceil} c_{wi}^2 \binom{i}{k-i} \binom{k}{i} \\
%	&\qquad + 2 \sum_{j 
%	= \big\lceil \frac{k+1}{2} \big\rceil}^{\min\{k, \lceil d/2 \rceil\}} 
%	\sum_{i= k-j}^{j-1} \sum_{w =1}^{\lceil d/2 \rceil}c_{wi} c_{wj} 
%	\binom{i}{k-j} \binom{k}{i} 
%	\quad \forall 2 \leq k \leq d, \nonumber
%	\end{align}
\begin{subequations}
	\label{eq:linSysEqCertComplete}
	\begin{align}
	c_{w,0} &= - c_{w,1} && \forall 1 \leq w \leq \lceil d/2 \rceil  \\
\label{eq:linSysEqCert}	
(-1)^k \bigg(\sum_{w=1}^{\lceil d/2 \rceil} c_{w,1}^2 + 1 \bigg) 
&= 
\sum_{i = \lceil k/2 \rceil}^{\min\{k, \lceil d/2 \rceil\}} \sum_{w 
	=1}^{\lceil d/2 \rceil} c_{w,i}^2 \binom{i}{k-i} \binom{k}{i} \\
 + 2 \sum_{j 
	= \big\lceil \frac{k+1}{2} \big\rceil}^{\min\{k, \lceil d/2 \rceil\}} 
&\sum_{i= k-j}^{j-1} \sum_{w =1}^{\lceil d/2 \rceil}c_{w,i} c_{w,j} 
\binom{i}{k-j} \binom{k}{i} 
 &&\forall 2 \leq k \leq d, \nonumber
\end{align}	
	\end{subequations}
	then
	for any choice of the vertex~$(g,h) \in V(\Bgr)$ the polynomials
	\begin{align*}
	s_{g^*h^*} &= x_{g^*h^*}  && \text{ for all } (g^*,h^*) \in 
	V(\Bgr)\setminus T_{gh}
	\text{ and}\\
	s_w &= \sum_{i=0}^{\lceil d/2 \rceil} c_{w,i} \rho_{gh}^i && \text{ for all 
	} 
	1 \leq w \leq \lceil d/2 \rceil
	\end{align*}
	form a $\lceil d/2 \rceil$-SOS-certificate of $f_\viz$, and therefore 
	Vizing's conjecture holds on the graph classes~$\G$ and~$\H$.
	
	%, if the coefficients~$c_{wi} \in \R$ for~$1 \leq w \leq \lceil d/2 \rceil$ and~$0 \leq i \leq \lceil d/2 \rceil$ are solution of the following system of equations
\end{theorem}

Note that the SOS-certificates of Theorem~\ref{thm:general-cert} are 
of the smallest possible degree according to 
Theorem~\ref{thm:min-degree-cert}.  
%Note that the polynomial $\rho_{gh}^0$ 
%is the constant 1 by Definition~\ref{def:rho}.
%\begin{remark}
Additionally, note that the system of 
equations~\eqref{eq:linSysEqCertComplete} depends on $d=n_\G 
+ n_\H - 
1$ and not on $n_\G$ or~$n_\H$ explicitly.
This means that if we find a solution for some~$d$, then we have found 
certificates for all graph classes~$\G$ and~$\H$ with $n_\G + n_\H - 1 = d$.
%\end{remark}
Furthermore, it can be observed that 
the constant terms in the 
polynomials~$s_w$ have to be the negative coefficients of the monomials of 
degree~$1$, 
as $c_{w,0}$ is the coefficient of~$\rho_{gh}^0 = 1$  and~$c_{w,1}$ 
is the coefficient of~$\rho_{gh}^1$ in~$s_w$.

%\begin{remark}
	Furthermore, observe that the system of equations~\eqref{eq:linSysEqCert} 
	coincides with 
	those for $n_\G = 3$, $n_\H = 2$, $k_\G = k_\H = 1$ (so $d=4$) in 
	Remark~\ref{rmk:cert32} and for $n_\G = n_\H = 2$, $k_\G=k_\H=1$ (so $d=3$)
	in Remark~\ref{rmk:cert2121}.
	So for these graph classes we were able to find a solution 
	of~\eqref{eq:linSysEqCertComplete}.
%\end{remark}

To prove Theorem~\ref{thm:general-cert}, we first consider some useful lemma.
\begin{lemma}
	\label{lemma:coeffs_rho_ssquared}
	Let~$d = n_\G + n_\H - 1$ and fix some vertex~$(g,h) \in V(\Bgr)$.
	Furthermore, let~$c_i \in \mathbb{R}$ for $0 \leq i \leq \lceil d/2 \rceil$ 
	define the polynomial $s \in P_\Bgr$  as
	\begin{equation*}
	s = \sum_{i = 0}^{\lceil d/2 \rceil} c_i \rho_{gh}^i.
	\end{equation*}
	Then $s$ squared is congruent to
	\begin{equation*}
	%s^2 \equiv
	\sum_{k=0}^{d} \Bigg( \sum_{i = \lceil k/2 \rceil}^{\min \{k, \lceil d/2 
	\rceil\}} c_i^2 \binom{i}{k-i} \binom{k}{i}+ 2 \sum_{j = \big\lceil 
	\frac{k+1}{2} \big\rceil}^{\min\{k, \lceil d/2 \rceil\}} \sum_{i= 
	k-j}^{j-1} 
	c_i c_j \binom{i}{k-j} \binom{k}{i} \Bigg) \rho_{gh}^k 
	%\quad \mod I_\viz
	\end{equation*}
%		\todo[inline]{@MS: why is the ordering in the second summation $j$, $i$ 
%		and 
%		not $i$, $j$? It would be more intuitive to go from order $i,j,k$ to 
%		order $k,i,j$, 
%		don't you think? Or does this ordering have to to with the fact that in 
%		this way the $i$ from the first sum and the $j$ from the second sum 
%		have almost the same boundaries?
%	
%		For the ordering $k,i,j$ I would get the boundaries 
%		
%		$0 <= k <= d$, 
%		
%		$\max\{k- \lceil d/2 \rceil, 0\} <= i <= \min \{k-1, \lceil d/2 
%		\rceil-1\}$,
%		
%		$\max\{k-i,i+1 \} <= j <= \min\{k, \lceil d/2 \rceil \}$.
%		
%		Can you please double check that?
%		
%		MS: $0 \leq i < j \leq \lceil d/2 \rceil$; I think the problem was that 
%$i < j$ is easier to represent by bounding the larger $j$ first and then $i$ 
%by 
%the value of $j$.
%}
	modulo $I_\viz$.
\end{lemma}
\begin{proof}
	By expanding the square of the polynomial~$s$, we get that
	\begin{equation*}
	s^2 = \sum_{i=0}^{\lceil d/2 \rceil} c_i^2 \rho_{gh}^i \rho_{gh}^i + 2\sum_{i=0}^{\lceil d/2 \rceil} \sum_{j=i+1}^{\lceil d/2 \rceil} c_i c_j \rho_{gh}^i \rho_{gh}^j.
	\end{equation*}
	Next, we use Corollary~\ref{cor:reduce-prod-polynomials} yielding  
	that
	\begin{equation*}
	%\label{eq:6.2-reduction-rho}
	\rho^i_{gh} \rho^j_{gh} \equiv \sum_{r = 0}^{\min \{i, d - j\}} 
	\binom{i}{r} \binom{j+r}{i} \rho^{j+r}_{gh}  = \sum_{k = j}^{\min \{i + j, 
	d\}} \binom{i}{k-j} \binom{k}{i} \rho^{k}_{gh} \quad \mod I_{\viz}
	\end{equation*}	
	holds for all $0 \leq i \leq j 
	\leq \lceil d/2 \rceil$.
	We apply this  
	to~$\rho_{gh}^i\rho_{gh}^j$ for~$0 \leq i \leq j \leq \lceil 
	d/2 \rceil$ and sum up the coefficients of~$\rho_{gh}^k$ for each~$k$ 
	with~$0 \leq k \leq d$. %to get the desired representation in 
	%Lemma~\ref{lemma:coeffs_rho_ssquared}.
	%One can see that 
	From the product  $c_i c_j \rho_{gh}^i  \rho_{gh}^j$ we get a contribution 
	of
	\begin{equation}
	\label{eq:cicj-contributiontok}
	c_i c_j \binom{i}{k-j} \binom{k}{i}
	\end{equation}	
	to the coefficient of $\rho_{gh}^k$ if $k$ is between $j$ and the minimum 
	of~$i+j$ and~$d$.
	For~$i=j$, this means that $i$ has to be between $k/2$ and $k$ and additionally, $i$ is less or equal to $\lceil d/2 \rceil$.
	In the case of~$i < j$, combining the inequalities~$k \leq j + i$ and~$i \leq j - 1$, we get that the inequalities~$j \geq (k+1)/2$ and~$i \geq k - j$ have to hold.
	Moreover, it holds that~$j \leq k$ and~$j \leq \lceil d/2 \rceil$.
	Therefore, collecting all coefficients of~$\rho_{gh}^k$ yields the stated 
	result.
\end{proof}
In the next corollary, we apply Lemma~\ref{lemma:coeffs_rho_ssquared} to the sum of all $s_w^2$ for $1 \leq w \leq \lceil d/2 \rceil$. 
\begin{corollary}
	\label{cor:sos-reduced-coeffsrho}
	Let $d = n_\G + n_\H - 1 $ and fix $(g,h) \in V(\Bgr)$. 
	Furthermore, let $c_{w,i} \in \R$ for $0 \leq i \leq \lceil d/2 \rceil$ and 
	for $1 \leq w \leq \big\lceil d/2 \big\rceil$ 
	define the polynomial $s_w$ as
	\[s_w = \sum_{i=0}^{\lceil d/2 \rceil} c_{w,i} \rho_{gh}^i. \]
	Then the sum of all polynomials $s_w$ squared is congruent to  
	\begin{equation*}
	\sum_{k=0}^{d} \Bigg( \sum_{i = \lceil k/2 \rceil}^{\min\{k, \lceil d/2 
	\rceil\}} \sum_{w =1}^{\lceil d/2 \rceil} c_{w,i}^2 \binom{i}{k-i} 
	\binom{k}{i} + 2 \sum_{j = \big\lceil \frac{k+1}{2} \big\rceil}^{\min\{k, 
	\lceil d/2 \rceil\}} \sum_{i= k-j}^{j-1} \sum_{w =1}^{\lceil d/2 
	\rceil}c_{w,i} c_{w,j} \binom{i}{k-j} \binom{k}{i} \Bigg) \rho_{gh}^k 
	\end{equation*}
	modulo $I_\viz$.
\end{corollary}
%\begin{proof}
%	Applying Lemma~\ref{lemma:coeffs_rho_ssquared} to each polynomial $s_w^2$ 
%yields the congruence
%	{\small
%		\begin{align*}
%		\sum_{w =1}^{\lceil d/2 \rceil} s_w^2 &\equiv \sum_{w =1}^{\lceil d/2 
%		\rceil}	\sum_{k=0}^{d} \Bigg( \sum_{i = \lceil k/2 \rceil}^{\min\{k, 
%		\lceil d/2 \rceil\}} c_{wi}^2 \binom{i}{k-i} \binom{k}{i} + 2 \sum_{j = 
%		\big\lceil \frac{k+1}{2} \big\rceil}^{\min\{k, \lceil d/2 \rceil\}} 
%		\sum_{i= k-j}^{j-1} c_{wi} c_{wj} \binom{i}{k-j} \binom{k}{i} \Bigg) 
%		\rho_{gh}^k\\
%		&= \sum_{k=0}^{d} \Bigg( \sum_{i = \lceil k/2 \rceil}^{\min\{k, \lceil 
%		d/2 \rceil\}} \sum_{w =1}^{\lceil d/2 \rceil} c_{wi}^2 \binom{i}{k-i} 
%		\binom{k}{i}\\
%		& \hspace*{3.8cm} + 2 \sum_{j = \big\lceil \frac{k+1}{2} 
%		\big\rceil}^{\min\{k, \lceil d/2 \rceil\}} \sum_{i= k-j}^{j-1} \sum_{w 
%		=1}^{\lceil d/2 \rceil}c_{wi} c_{wj} \binom{i}{k-j} \binom{k}{i} \Bigg) 
%		\rho_{gh}^k \quad \mod I_\viz.
%		\end{align*}%
%	}%
%\end{proof}

Finally, we have all ingredients to prove Theorem~\ref{thm:general-cert}.
\begin{proof}[Proof of Theorem~\ref{thm:general-cert}]
	The proof is analogous to the ones of Theorem~\ref{thm:cert32} 
	and~\ref{thm:cert2121}.
	%Let $c_{wi} \in \R$ for~$1 \leq w \leq \lceil d/2 \rceil$ and~$0 \leq i 
	%\leq 
	%\lceil 
	%d/2 \rceil$ be a real solution to the system of equations in 
	%Theorem~\ref{thm:general-cert}.
	First, we fix a vertex $(g,h) \in V(\Bgr)$.
	For brevity, we write $\rho^k$ for $\rho_{gh}^k$ for all $0 \leq k \leq \lceil d/2 \rceil$.
	Next, we use the fact that $x_{g^*h^*}^2 \equiv x_{g^*h^*} \ \mod I_\viz$, to get the congruence
	\begin{align*}
	\sum_{(g^*,h^*) \in V(\Bgr)\setminus T_{gh}} x_{g^*h^*}^2 + \rho^1 - 1 &\equiv  \sum_{(g^*,h^*) \in V(\Bgr)\setminus T_{gh}} x_{g^*h^*} + \rho^1 - 1 \\
	& = \sum_{(g^*,h^*) \in V(\Bgr)\setminus T_{gh}} x_{g^*h^*} + \sum_{(g',h') \in T_{gh}} x_{g'h'}  - 1 \\
	%&= \sum_{(g,h) \in V(\Bgr)} x_{gh} - 1 \\
	& = f_\viz \quad \mod I_\viz.
	\end{align*}
	The above and Corollary~\ref{cor:sos-reduced-coeffsrho} yield the congruence
	\begin{multline}
	\label{eq:proof-cert-general}
	\small
	\sum_{(g^*,h^*) \in V(\Bgr)\setminus T_{gh}} x_{g^*h^*}^2  + 
	\sum_{w=1}^{\lceil d/2 \rceil} s_w^2 \equiv  f_\viz - \rho^1 + 1 
	+ 
	\sum_{k=0}^{d} \Bigg( \sum_{i = \lceil k/2 \rceil}^{\min\{k, \lceil d/2 
	\rceil\}} \sum_{w =1}^{\lceil d/2 \rceil} c_{w,i}^2 \binom{i}{k-i} 
	\binom{k}{i}\\
	+ 2 \sum_{j = \big\lceil \frac{k+1}{2} \big\rceil}^{\min\{k, \lceil d/2 
	\rceil\}} \sum_{i= k-j}^{j-1} \sum_{w =1}^{\lceil d/2 \rceil}c_{w,i} c_{w,j} 
	\binom{i}{k-j} \binom{k}{i} \Bigg) \rho^k \ \mod I_\viz.
	\end{multline}
	By writing down the coefficients of $\rho^0 = 1$ and $\rho^1$ in~\eqref{eq:proof-cert-general} explicitly, we get that the sum of squares is congruent to
	\begin{multline*}
	f_\viz - \rho^1 + \rho^0 + \bigg(\sum_{w = 1}^{\lceil d/2 \rceil} 
	c_{w,0}^2\bigg)\rho^0 + \bigg(\sum_{w= 1}^{\lceil d/2 \rceil} (c_{w,1}^2 + 
	2c_{w,0}c_{w,1})\bigg) \rho^1 \\
	+ \sum_{k=2}^{d} \Bigg( \sum_{i = \lceil k/2 \rceil}^{\min\{k, \lceil d/2 
	\rceil\}} \sum_{w =1}^{\lceil d/2 \rceil} c_{w,i}^2 \binom{i}{k-i} 
	\binom{k}{i}
	\\
	+ 2 \sum_{j = \big\lceil \frac{k+1}{2} \big\rceil}^{\min\{k, \lceil d/2 
	\rceil\}} \sum_{i= k-j}^{j-1} \sum_{w =1}^{\lceil d/2 \rceil}c_{w,i} c_{w,j} 
	\binom{i}{k-j} \binom{k}{i} \Bigg) \rho^k
	\end{multline*}
	modulo~$I_\viz$.
	If the coefficients $c_{w,i}$ satisfy $ c_{w,0} = - c_{w,1}$ 
	(and thus also  $c_{w,1}^2 + 2c_{w,0}c_{w,1} = -c_{w,1}^2$)
	and  
	 the system of 
	equations~\eqref{eq:linSysEqCert}, then the above expression equals
	\begin{equation*}
	f_\viz + \bigg( \sum_{w = 1}^{\lceil d/2 \rceil} c_{w,1}^2 + 1\bigg) 
	\bigg(\sum_{k= 0}^d (-1)^k \rho^k\bigg),
	\end{equation*}
	which is congruent to $f_\viz$
	modulo $I_\viz$ due to Lemma~\ref{lemma:reduction-polynomial-in-ideal} and 
	Lemma~\ref{lemma:reduction-polynomial-as-sum}.
	Hence, the polynomials stated in Theorem~\ref{thm:general-cert} form a 
	$\lceil d/2 \rceil$-SOS-certificate for the graph classes~$\G$ and~$\H$ if 
	the coefficients~$c_{w,i}$ fulfill~\eqref{eq:linSysEqCertComplete}.	
\end{proof}

To summarize, Theorem~\ref{thm:general-cert} states that if we find a solution 
to the system of equations~\eqref{eq:linSysEqCertComplete}, then we obtain an 
SOS-certificate of minimum degree.
In fact, it can also be deduced that if there is an SOS-certificate 
of the form given by Theorem~\ref{thm:general-cert}, then  the system of 
equations~\eqref{eq:linSysEqCertComplete} has to hold.

\subsection{Finding a Solution of the System of Equations}

Next, we consider the problem of finding such solutions.
Towards that end, 
let the vector $c_i$ be defined as
\begin{equation*}
c_i = (c_{w,i})_{1 \leq w \leq \lceil d/2 \rceil} 
\end{equation*}
for all $0 \leq i \leq \lceil d/2 \rceil$, 
so~$c_i$ denotes the vector collecting all coefficients of $\rho_{gh}^i$ in the 
general SOS-certificate for $k_\G = k_\H = 1$ and $d = n_\G + n_\H - 1$ stated 
in Theorem~\ref{thm:general-cert}.
It is easy to see that the system of equations~\eqref{eq:linSysEqCertComplete} 
has a 
solution if and only if there are vectors $c_i \in \R^{\lceil d/2 \rceil}$ for 
$0 \leq i \leq \lceil d/2 \rceil$ that are a solution to the  
system of equations 
\begin{subequations}
	\label{eq:linSysEqCertDotProd}
\begin{align}
	c_0 &= -c_1 && \forall 1 \leq w \leq \lceil d/2 \rceil  \\
%-(\langle c_0,c_0 \rangle + 1 )
%&= -1 + \langle c_1,c_1 \rangle + 2\langle c_0,c_1 \rangle 	
%\quad \text{ and }\\
(-1)^k (\langle c_1,c_1 \rangle  + 1) &= 
\sum_{i = \lceil k/2 \rceil}^{\min\{k, \lceil d/2 \rceil\}} \langle c_i, c_i 
\rangle \binom{i}{k-i} \binom{k}{i} \\
& + 2 \sum_{j = \big\lceil \frac{k+1}{2} 
\big\rceil}^{\min\{k, \lceil d/2 \rceil\}} \sum_{i= k-j}^{j-1} \langle c_i, c_j 
\rangle \binom{i}{k-j} \binom{k}{i}
&& \forall 2 \leq k \leq d. \nonumber
\end{align}
\end{subequations}

%The first equation is equivalent to
%\begin{equation*}
%\langle c_0 + c_1, c_0 + c_1 \rangle = 0,
%\end{equation*}
%which yields that $c_0 = -c_1$ has to hold and hence
%$c_{w0} = -c_{w1}$ holds for all $1 \leq w\leq \lceil d/2 \rceil$.
%\begin{remark}

	%In further consequence, this means that the
	% coefficients of $\rho_{gh}^{2m}$ have to be equal to $\langle c_1 , c_1 \rangle + 1$ and the ones of $\rho_{gh}^{2m + 1}$ have to be equal to $- \langle c_1,c_1 \rangle - 1$.
%\end{remark}

The system of equations~\eqref{eq:linSysEqCertDotProd} can be rewritten 
using~$F_{i,j} = \langle c_i , c_j \rangle$ for $0 \leq i,j\leq \lceil d/2 
\rceil$. Let~$F$ be the~$\lceil d/2 \rceil \times \lceil d/2 \rceil$-matrix 
with $F = (F_{i,j})_{1 \leq i,j\leq \lceil d/2 \rceil}$, i.e., $F$ is the Gram 
matrix of the matrix $C = (c_{w,i})_{1 \leq w,i \leq \lceil d/2 \rceil}$ and 
$F = C^\transposed C$  holds.
 As any 
Gram matrix is positive semidefinite and any positive semidefinite matrix is 
the Gram matrix of some set of vectors (which can for example be determined 
using Cholesky decomposition), we obtain the following result.

%Now, let~$D_{ij} = \langle c_i , c_j \rangle$ for $1 \leq i,j\leq \lceil d/2 
%\rceil$, then the system of equations we need to solve is a system of linear 
%equations in the variables~$D_{ij}$.
%In total, we have $d - 1$ linear equations in~$\lceil d/2 \rceil  (\lceil d/2 
%\rceil + 1) /2$ variables~$D_{ij}$ with $D_{ij} = D_{ji}$ for $1 \leq i , j 
%\leq d$.
%
%
%To obtain a certificate, we need to find coefficients $c_{wi}$ for $1 \leq i,w 
%\leq \lceil d/2 \rceil$ such that
%$ D_{ij} = \langle c_i, c_j \rangle$ for all $1 \leq i,j \leq \lceil d/2 
%\rceil$ and the entries of $D_{ij}$ fulfill the above linear equations.
%Clearly, if there exist such coefficients, the matrix~$D$ is positive 
%semidefinite, as it can be written as
%$D = C^\transposed C$ with $C$ being the matrix with entries~$(C)_{w,i} = 
%c_{wi}$ for all $1 \leq w,i \leq \lceil d/2 \rceil$. 

\begin{observation}
	\label{obs:SDPforCert}
	The system of equations~\eqref{eq:linSysEqCertComplete} has a real solution 
	if and 
	only 
	if there is a positive semidefinite matrix $F = (F_{i,j})_{1 \leq i,j\leq 
	\lceil d/2 \rceil}$ such that 
\begin{align}
\label{eq:sysOfEquInF}
(-1)^k (F_{1,1} + 1) = 
\sum_{i = \lceil k/2 \rceil}^{\min\{k, \lceil d/2 \rceil\}} F_{i,i} 
\binom{i}{k-i} \binom{k}{i} 
 + 2 \sum_{j = \big\lceil \frac{k+1}{2} 
	\big\rceil}^{\min\{k, \lceil d/2 \rceil\}} \sum_{i= k-j}^{j-1} F_{i,j}  
	\binom{i}{k-j} \binom{k}{i},
\end{align}
where we substitute $F_{0,j} = -F_{1,j}$, 
holds for all  $2 \leq k \leq d$.
In particular, 
$F = C^\transposed C$ for 
$C = (c_{w,i})_{1 \leq w,i \leq \lceil d/2 \rceil}$
and $c_{w,0} = -c_{w,1}$ for all $1 \leq w\leq \lceil d/2 \rceil$  holds for 
corresponding 
solutions. 	
\end{observation}

With Observation~\ref{obs:SDPforCert} we have transformed the task of finding a 
certificate from solving a system of 
$\lceil d/2 \rceil$ linear and $d-1$ quadratic 
equations~\eqref{eq:linSysEqCertComplete} in
$\lceil d/2 \rceil \left(\lceil d/2 \rceil + 1\right)$ variables to solve an 
SDP 
with matrix variable of dimension $\lceil d/2 \rceil$ with $d-1$ linear 
equality constraints.

The objective function of this SDP can be chosen arbitrarily, as any 
feasible solution leads to an SOS-certificate.
Unfortunately, just solving this SDP with an off-the-shelf SDP solver is not 
enough because 
any feasible solution obtained from an SDP solver is 
numerical, i.e., the system of equations is not fulfilled exactly, but only 
with small numerical errors.
So in order to find a certificate, there is still some lucky guessing required.

Thus, we follow a different road to find a positive semidefinite matrix $F$ 
that is an exact solution to the system of linear 
equations~\eqref{eq:sysOfEquInF}. In fact, any solution $F$ to the system of 
equations~\eqref{eq:sysOfEquInF} can be represented as linear expression in 
some free variables, which 
are a subset of all variables $F_{i,j}$. We iteratively fix the free 
variables by solving SDPs in the following way. 
When we want to fix the free variable $F_{i,j}$, we solve the SDP  
with matrix variable $F$, the system of linear 
equations~\eqref{eq:sysOfEquInF} and the already fixed free variables two 
times, one time with maximizing and one 
time with minimizing the value of the free variable $F_{i,j}$ as objective 
function. Let $F_{i,j}^{min}$ and $F_{i,j}^{max}$ denote the optimal objective 
function values of these SDPs. We fix the free variable $F_{i,j}$ to an 
arbitrary rational number in the interval $[F_{i,j}^{min}, F_{i,j}^{max}]$,
where we try to set $F_{i,j}$ to a rational number with small denominator in 
order to obtain ``nice'' values in $F$. Then we proceed with the next free 
variable. 

Clearly, the choice of $F_{i,j}$ in the interval $[F_{i,j}^{min}, F_{i,j}^{max}]$ 
makes sure that we find a positive semidefinite matrix $F$ 
that is an exact solution to the system of linear 
equations~\eqref{eq:sysOfEquInF} with this procedure if it exists. 
If the system of linear equations~\eqref{eq:sysOfEquInF} has no positive 
semidefinite solution, then we are not able to find 
a certificate of the specific form stated in Theorem~\ref{thm:general-cert}.

As already mentioned before Observation~\ref{obs:SDPforCert}, from $F$ we can 
obtain the coefficients $c_{w,i}$ of the SOS-certificate from 
Theorem~\ref{thm:general-cert} with a Cholesky decomposition.
%
%But if the system of linear equations has a solution, the result gives us some 
%entries of $D$, which are considered as free parameters, and all other entries 
%can be .
%
%As a result, we are able to find concrete values for the entries in $D$ more 
%easily.
%In particular, we need to find exact values for all entries representing free 
%parameters, such that the resulting matrix~$D$ is positive semidefinite.
%We do so by trying to fix these entries iteratively.
%The following observation will help us to determine an interval of values for 
%which the resulting matrix when setting an entry to this value is still 
%positive semidefinite.
%\begin{observation}
%	\label{lemma:combinationSDPSolutions-psd}
%	Let $a_1$, $a_2 \in \R$ and $X$, $Y_1$ and $Y_2$ be matrices such that
%	\begin{align*}
%	a_1 X + Y_1 & \succcurlyeq 0 \text{ and} \\
%	a_2 X + Y_2 & \succcurlyeq 0,
%	\end{align*}
%	then 
%	\[ \big(\xi a_1 + (1-\xi) a_2 \big) X + \xi Y_1 + (1-\xi) Y_2 \succcurlyeq 
%0 \]
%	holds for all $\xi \in (0,1)$.
%\end{observation}
%%\begin{proof}
%%	The observation follows from the fact that the sum of two positive 
%%semidefinite matrices is again positive semidefinite and that a positive 
%%semidefinite matrix multiplied with a positive factor is again positive 
%%semidefinite.
%%\end{proof}
We now consider an example to demonstrate our approach to determine an 
SOS-certificate.

%\newcounter{examplecounter}
%\setcounter{examplecounter}{\value{theorem}}
\begin{example}
	\label{ex:4121}
	Let~$n_\G = 4$, $n_\H = 2$ and~$k_\G = k_\H = 1$, so $d=5$.
	To find a certificate as stated in Theorem~\ref{thm:general-cert}  
	we need to find a $3\times 3$ positive semidefinite matrix~$F$ such that 
	its entries satisfy the system of equations~\eqref{eq:sysOfEquInF}, i.e., 
	the equations
	\begin{align}
	\begin{split}
	\label{algn:matrixEquations-example}
	F_{1,1} + 1 &= 2 F_{1,1} + 2 F_{2,1} + F_{2,2}, \\
	- (F_{1,1} + 1)&= 6 F_{2,1} + 6 F_{2,2} + 4 F_{3,1} + 6 F_{3,2} + F_{3,3}, \\
	F_{1,1} + 1 &= 6 F_{2,2} + 8 F_{3,1} + 24 F_{3,2} + 12 F_{3,3} \text{\quad \ 
		and}\\
	- (F_{1,1} + 1)&=20 F_{3,2} + 30 F_{3,3}.
	\end{split}	
	\end{align}
	All possible solutions of this system of linear equations can be 
	written as
	\begin{small} 
	\begin{align}
	\begin{split}
	\label{algn:matrixConstraints-example}	
	F_{2,1} &= -\frac{1}{2} F_{1,1} - \frac{1}{2} F_{2,2} + \frac{1}{2}, \\
	F_{3,1} &= \frac{47}{40} F_{1,1} - \frac{3}{4} F_{2,2} - \frac{133}{40}, \\
	F_{3,2} &= -\frac{1}{2} F_{1,1} + \frac{7}{4} \text{\quad and} \\
	F_{3,3} &= \frac{3}{10} F_{1,1} - \frac{6}{5},
	\end{split}
	\end{align}
\end{small}
	where $F_{1,1}$ and $F_{2,2}$ are free parameters.
%\end{example}
%\newcounter{helpcounter}
%\setcounter{helpcounter}{\value{theorem}}
%\setcounter{theorem}{\value{examplecounter}}
%\begin{example}[continued]
	Thus, we can write any matrix $F$, which 
	fulfills~\eqref{algn:matrixConstraints-example} and 
	hence~\eqref{algn:matrixEquations-example}, as
	\begin{footnotesize}
	\begin{equation}
	\label{eq:termOfConstraint-F-psd}
	F_{1,1}
	\begin{pmatrix}
	1 & -1/2 & 47/40 \\
	-1/2 & 0 & -1/2 \\
	47/40 & -1/2 & 3/10 
	\end{pmatrix}
	+ F_{2,2}
	\begin{pmatrix}
	0 & -1/2 & -3/4 \\
	-1/2 & 1 & 0 \\
	-3/4 & 0 & 0
	\end{pmatrix}
	+ \begin{pmatrix}
	0 & 1/2 & -133/40 \\
	1/2 & 0 & 7/4 \\
	-133/40 & 7/4 & -6/5
	\end{pmatrix}
	\end{equation}
\end{footnotesize}
	for the free variables $F_{1,1}$ and $F_{2,2}$.
	Next, we need to find exact values for $F_{1,1}$ and $F_{2,2}$ such that the 
	resulting matrix $F$ is positive semidefinite.
	
	Let~$F_{1,1}^{min}$ be the result of the SDP which minimizes~$F_{1,1}$ under 
	the 
	constraint that~\eqref{eq:termOfConstraint-F-psd} is positive semidefinite.
	Furthermore, let $F_{1,1}^{max}$ be the optimal solution of the same SDP 
	with an 
	objective that maximizes~$F_{1,1}$.
	For this example we get the (numerical) optimal solutions~$F_{1,1}^{min} = 
4.68455$ and~$F_{1,1}^{max} = 38.41658$.	
	We can set 
	$F_{1,1}$ to be any rational value in the interval 
	$[F_{1,1}^{min},F_{1,1}^{max}]$ %as 
	and choose $F_{1,1}= 6$. 
	
	As a consequence, we get that $F$ has to be of the form
	\begin{footnotesize}
	\begin{equation*}
	F_{2,2}
	\begin{pmatrix}
	0 & -1/2 & -3/4 \\
	-1/2 & 1 & 0 \\
	-3/4 & 0 & 0
	\end{pmatrix}
	+ 
	\begin{pmatrix}
	6 & -5/2 & 149/40 \\
	-5/2 & 0 & -5/4 \\
	149/40 & -5/4 & 3/5
	\end{pmatrix}.
	\end{equation*}
\end{footnotesize}
	To find a rational value for $F_{2,2}$, we follow the same strategy.
	We determine $F_{2,2}^{min} = 
	2.64289$ and $F_{2,2}^{max} = 3.26414$ and 
	choose~$F_{2,2} = 3$ and finally obtain the matrix 
	\begin{equation*}
	F = 
	\begin{pmatrix}
	6 & -4 & 59/40 \\
	-4 & 3 & -5/4 \\
	59/40 & -5/4 & 3/5
	\end{pmatrix},
	\end{equation*}
	which is positive semidefinite and fulfills the system of 
	equations~\eqref{algn:matrixEquations-example} exactly. 
To determine the solution of the system of 
equations~\eqref{eq:linSysEqCertComplete}, i.e.,\ the coefficient matrix~$C$, we 
compute the Cholesky factorization of~$F = C^\transposed C$ and obtain
\begin{equation*}
C = 
\begin{pmatrix}
\sqrt{6} & -2/3 \sqrt{6} & 59/240 \sqrt{6}\\
0 & 1/3 \sqrt{3} & -4/15 \sqrt{3} \\
0 & 0 & 1/80 \sqrt{154}
\end{pmatrix}.
\end{equation*}	
	
\end{example}
%\setcounter{theorem}{\value{helpcounter}}

%As soon as we have found a rational positive semidefinite matrix~$D$ like 
%demonstrated in the above example, we can compute its Cholesky 
%factorization~$C^\transposed C = D$ and obtain a solution to the system of 
%equations~\eqref{eq:linSysEqCert} and hence 
%a~$\lceil d/2 \rceil$-SOS-certificate.

%\begin{remark}
%	Since~$C$ is an upper triangular matrix, it holds that all coefficients 
%$c_{wi}$ with $1 \leq i <  w \leq \lceil d/2 \rceil $ are equal to 0.
%\end{remark}

%We go back to Example~\ref{ex:4121} one more time to compute the matrix~$C$.
%\setcounter{helpcounter}{\value{theorem}}
%\setcounter{theorem}{\value{examplecounter}}
%\begin{example}[continued]

%\end{example}
%\setcounter{theorem}{\value{helpcounter}}

As a consequence of Example~\ref{ex:4121} and Theorem~\ref{thm:general-cert}, 
we have found the following 
$3$-SOS-certificate for $n_\G = 4$ and $n_\H = 2$ as well as for $n_\G = n_\H = 
3$ and $k_\G = k_\H = 1$.

\begin{corollary}
	\label{thm:cert-d6}
	Let $\G$ and $\H$ be two graph classes with $d = n_\G + n_\H - 1 = 5$ and $k_\G=k_\H=1$, then Vizing's conjecture is true for these graph classes, as for any vertex $(g,h)\in V(\Bgr)$ the polynomials
	\begin{align*}
	s_{g^*h^*} &= &x_{g^*h^*} && && && && {\small \text{for all } (g^*,h^*) \in 
	V(\Bgr)\setminus T_{gh}}\\
	s_1 &= & - \sqrt{6} &&+ \sqrt{6} \rho_{gh}^1  &&-\frac{2}{3} \sqrt{6} 
	\rho_{gh}^2  &&+ \frac{59}{240} \sqrt{6} \rho_{gh}^3,\\
	s_2 &=  & && && \frac{1}{3} \sqrt{3} \rho_{gh}^2 && - \frac{4}{15} \sqrt{3} 
	\rho_{gh}^3   &&\text{and}\\
	s_3 &= & && && &&\frac{1}{80} \sqrt{154} \rho_{gh}^3
	\end{align*}
	form a $3$-SOS-certificate of $f_\viz$.
\end{corollary}

%
%To sum it up, the following steps are performed to find a solution to the 
%system of equations in Theorem~\ref{thm:general-cert} for $d = n_\G + n_\H - 
%1$. 
%\begin{enumerate}
%	\item Set $c_{w0} = - c_{w1}$ for all $1 \leq w \leq \lceil d/2 \rceil$ and 
%rewrite the system of equations stated in Theorem~\ref{thm:general-cert} in 
%terms of the variables $D_{ij} = \sum_{w=1}^{\lceil d/2 \rceil} c_{wi}c_{wj}$ 
%for all $i$, $j$ with $1 \leq i,j \leq \lceil d/2 \rceil$ to obtain a system 
%of 
%linear equations.
%	\item Solve the  system of linear equations to obtain all solution of the 
%system in dependence on the free parameters.
%	\item Find exact values for all free parameters in the solution of Step 2 
%such that the matrix~$D$ with entries $(D)_{ij} = D_{ij}$ is positive 
%semidefinite. To do so, apply the following steps iteratively.
%	\begin{enumerate}
%		\item Choose one entry that is a free parameter and find its 
%(numerical) lower and upper bound under the constraint that $D$ is positive 
%semidefinite by solving the SDPs maximizing and minimizing this free parameter.
%		\item Set this particular entry in $D$ to be a rational number between 
%the lower and upper bound and update all other entries in~$D$, which depend on 
%this entry.
%	\end{enumerate}
%	\item Compute $C$ such that $D = C^\transposed C$, for example Cholesky. 
%\end{enumerate}

\subsection{Theoretical Properties of Certificates}
It turns out that for even values of $d$ we can say more about the system of 
equations~\eqref{eq:sysOfEquInF}, in particular $F_{1,1}$ is fixed as stated in 
the following corollary.
\begin{corollary}
	\label{cor:fixedF11}
	Let $d \geq 4$ be even, then $F_{1,1} = d-1$ holds.	
\end{corollary}

\begin{proof}
	We show that $F_{1,1} = d-1$ holds by adding up the equations 
	of~\eqref{eq:sysOfEquInF} multiplied by $(-1)^k \frac{d}{k(k-1)}$ for all 
	$k$ with $2 \leq k \leq d$.
	
	On the left-hand side of the resulting equation we get
	\begin{equation*}
	d(F_{1,1} + 1)\sum_{k=2}^d \frac{1}{k(k-1)} = d (F_{1,1} + 1) \frac{d-1}{d} = (F_{1,1} + 1)(d-1).
	\end{equation*}
	The right-hand side is
	\begin{equation}
	\label{eq:rhs-sum-equations-in-F}
	\sum_{k = 2}^d \frac{(-1)^k d}{k(k-1)} \Bigg( \sum_{i = \lceil k/2 \rceil}^{\min\{k, \lceil d/2 \rceil\}} F_{i,i} 
	\binom{i}{k-i} \binom{k}{i} 
	+ 2 \sum_{j = \big\lceil \frac{k+1}{2} 
		\big\rceil}^{\min\{k, \lceil d/2 \rceil\}} \sum_{i= k-j}^{j-1} F_{i,j}  
	\binom{i}{k-j} \binom{k}{i} \Bigg).
	\end{equation}
	It is enough to show that~\eqref{eq:rhs-sum-equations-in-F} 
	equals~$dF_{1,1}$\revision{.}
	The variable $F_{1,1}$ appears only once in~\eqref{eq:rhs-sum-equations-in-F}, namely for $k = 2$ with the coefficient $\frac{d}{2}\binom{1}{1}\binom{2}{1} = d$.
	Thus, it remains to show that the coefficients of all other variables in~\eqref{eq:rhs-sum-equations-in-F} sum up to zero.

	We start with the variables~$F_{i,j}$ for $1 < i \leq j \leq d/2$.
	With the same arguments as in the proof of Lemma~\ref{lemma:coeffs_rho_ssquared} to obtain the bounds on $k$ for~\eqref{eq:cicj-contributiontok}, $F_{i,j}$ appears in all summands with $k$ between $j$ and $\min  \{d, i+j\} = i + j$.
	% Argument as in proof of Lemma 5.2 contribution to k for j <= k <= i+j
	Hence, the coefficient of $F_{i,j}$ in~\eqref{eq:rhs-sum-equations-in-F} for $i = j$ is
	\begin{equation}
	\label{eq:gosper-sum}
	\sum_{k=j}^{i+j}\frac{\left(-1\right)^{k} d{i \choose  k-j} {k \choose i}}{{\left(k - 1\right)} k}
	\end{equation}
	and for $i < j$ it is two times~\eqref{eq:gosper-sum}.
	It can be shown that~\eqref{eq:gosper-sum} is equal to zero.
	
	The variables left to consider are $F_{1,j}$ for $1 < j \leq d/2$. The variable $F_{1,j}$  appears only in the summands of~\eqref{eq:rhs-sum-equations-in-F} for $k = j$ and $k = j+1$. Moreover, the variable~$F_{0,j}$, which is equal to $-F_{1,j}$, appears in the summand with $k = j$ only.
	Therefore, when we substitute $F_{0,j} = -F_{1,j}$, the coefficient of $F_{1,j}$ in~\eqref{eq:rhs-sum-equations-in-F} is
	\begin{equation*} 
	(-1)^j \frac{2d}{j(j-1)} \Bigg( \binom{1}{0}\binom{j}{1} - 
	\binom{0}{0}\binom{j}{0} \Bigg) + (-1)^{j+1}\frac{2d}{(j+1)j}\binom{1}{1} 
	\binom{j+1}{1} = 0,
	%(-1)^j \frac{2d}{j} - (-1)^j \frac{2d}{j} = 0
	\end{equation*}
	which completes the proof.
\end{proof}

Corollary~\ref{cor:fixedF11} shows that for all even $d\geq4$, the left-hand 
sides of~\eqref{eq:sysOfEquInF} are fixed to  $(-1)^k d$.
This implies that for all certificates of the form stated in 
Theorem~\ref{thm:general-cert} for any fixed vertex $gh$, the sum of the 
polynomials squared and then reduced by the polynomials of degree 2 in the 
Gröbner basis stated in Theorem~\ref{thm:gb} equals 
\begin{equation*}
	f_\viz + d \sum_{k=0}^d (-1)^k \rho_{gh}^k,
\end{equation*}
which is congruent to $f_\viz$ modulo $I_\viz$.

Moreover, the fact that $F_{1,1}$ is fixed implies that $F_{\frac d 2, \frac d 2}$ and $F_{\frac d 2 - 1, \frac d 2}$ are fixed too.

%For $n_\G + n_\H - 1 \leq 14$ and even, our computations lead to the fact that 
%the coefficient $c_{11}$ has to be $\sqrt{n_\G + n_\H - 2}$ in any 
%SOS-certificate of the form given in Theorem~\ref{thm:general-cert}.

\begin{observation}
	For all even $d$ in~\eqref{eq:sysOfEquInF} the equation for $k = d$ is
	\begin{equation*}
	F_{1,1} =\binom{d}{d/2} F_{\frac{d}{2},\frac{d}{2}}  - 1
	\end{equation*}
	and the equation for $k = d-1$ is
	\begin{equation*}
	F_{1,1} = - \frac{d}{2}\binom{d-1}{d/2} F_{\frac{d}{2},\frac{d}{2}} - 2 \binom{d-1}{d/2 - 1} F_{\frac{d}{2}-1,\frac{d}{2}} - 1 %= \binom{d-1}{d/2} \bigg( -\frac{d}{2} F_{\frac{d}{2},\frac{d}{2}}  - F_{\frac{d}{2}-1,\frac{d}{2}} \bigg) - 1
	\end{equation*}
	Since $F_{1,1} + 1 = d$ by Corollary~\ref{cor:fixedF11}, this implies that
	\begin{align*}
	F_{\frac{d}{2},\frac{d}{2}} &= \frac{F_{1,1} + 1}{\binom{d}{d/2}}  = \frac{d}{\binom{d}{d/2}}\text{ and}\\
	F_{\frac{d}{2} - 1, \frac{d}{2}} & = -\frac{d + d^2 \binom{d-1}{d/2}/(2\binom{d}{d/2})} {2 \binom{d - 1}{d/2-1}} = - \frac{d + d^2/4}{\binom{d}{d/2}} = - F_{\frac{d}{2}, \frac{d}{2}} (1 + d/4)
	\end{align*}
	% MORE STEPS:
	%\begin{align*}
	%	\binom{d-1}{d/2} / \binom{d}{d/2} &= \frac{d/2}{d} \binom{d}{d/2} / \binom{d}{d/2} = 1/2 \text{ and therefore also}\\
	%	2 \binom{d-1}{d/2} &= \binom{d}{d/2} \\
	%	2\binom{d-1}{d/2 - 1}& = 2\frac{(d-1)!}{(d/2)!(d-1-d/2)!} = 2\binom{d-1}{d/2} = \binom{d}{d/2}	
	%\end{align*}
	holds.
\end{observation}

\subsection{Further Certificates for 
	\texorpdfstring{$k_\G = k_\H = 1$}{kG=kH=1}}
%\label{chapter:furthercert}

%\subsection{Implementation}

%Given the steps above, 
We implemented the above-described procedure 
to find an SOS-certificate as stated in 
Theorem~\ref{thm:general-cert} for Vizing's conjecture for the graph classes 
$\G$ and $\H$ satisfying 
$d = n_\G + n_\H - 1$ and $k_\G = k_\H = 1$
in SageMath~\cite{sagemath}.

\revision{
	The implementation is available as ancillary files from the arXiv page of this paper
	at \href{https://arxiv.org/src/2112.04007/anc}{arxiv.org/src/2112.04007/anc}.
}
In particular, the method \texttt{find\_certficate(d)} returns for a given 
integer $d$ the coefficient matrix $C$ of a $\lceil d/2 \rceil$-SOS-certificate 
for Vizing's 
conjecture.

With the help of this code, we were able to find 
SOS-certificates for Vizing's conjecture on the 
graph classes $\G$ and $\H$ with $k_\G = k_\H = 1$ and $d = n_\G + n_\H - 1$ 
for $6 \leq d \leq 14$.
%Theorem~\ref{thm:general-cert} ensures that the obtained coefficients indeed 
%determine SOS-certificates for Vizing's conjecture. 
\begin{corollary}
	For all graph classes $\G$ and $\H$ with $k_\G = k_\H = 1$ and $d = n_\G + 
	n_\H - 1$ with $6 \leq d \leq 14$ Vizing's conjecture is true, because the 
	polynomials 
	\begin{align*}
s_{g^*h^*} &= x_{g^*h^*}  && \text{ for all } (g^*,h^*) \in 
V(\Bgr)\setminus T_{gh}
\text{ and}\\
s_w &= \sum_{i=0}^{\lceil d/2 \rceil} c_{w,i} \rho_{gh}^i && \text{ for all 
} 
1 \leq w \leq \lceil d/2 \rceil
\end{align*}
	form a $\lceil d/2 \rceil$-SOS-certificate of $f_\viz$ for every choice of 
	$(g,h) \in V(\Bgr)$.
	Here
	$c_{w,i}= 0 $ for all  $1 \leq i <  w \leq \lceil d/2 \rceil $ 
	and
	$c_{w,0}= -c_{w,1}$ for all  $1 \leq w \leq \lceil d/2 \rceil $ 
	hold for all 
	values of~$d$. 
	Furthermore, 
	for $d=6$ we have
\begin{small}
\begin{alignat*}{3}
&c_{1,1} = \sqrt{5}, \qquad \qquad && c_{1,2} = -\frac{3}{5} \sqrt{5}, \qquad 
\qquad && c_{1,3} = \frac{21}{100} \sqrt{5},\\
& && c_{2,2} = \frac{1}{5} \sqrt{5}, && c_{2,3} = - \frac{3}{25} \sqrt{5} \quad 
\text{and}\\
& && &&  c_{3,3} = \frac{1}{20} \sqrt{3};
\end{alignat*}
\end{small}
	for $d=7$ we have
\begin{small}	
\begin{alignat*}{4}
&c_{1,1} = \sqrt{7}, \qquad  && c_{1,2} =  - \frac{5}{7} \sqrt{7},
\qquad && c_{1,3} = \frac{9}{28} \sqrt{7},\qquad \qquad && c_{1,4} = 
-\frac{17}{245} \sqrt{7},\\
& && c_{2,2} = \frac{1}{7} \sqrt{21}, && c_{2,3} = -\frac{179}{1260} \sqrt{21}, 
&& c_{2,4} = \frac{109}{2205} \sqrt{21},\\
& && &&  c_{3,3} = \frac{1}{90} \sqrt{429},&& c_{3,4} = - \frac{53}{6435} 
\sqrt{429} \quad \text{and}\\
& && && && c_{4,4} = \frac{1}{5005} \sqrt{4147};
\end{alignat*}
\end{small}
	
	for $d=8$ we have
\begin{small}	
\begin{alignat*}{4}
&c_{1,1} = \sqrt{7}, \qquad && c_{1,2} = -\frac{5}{7} \sqrt{7}, 
\qquad && c_{1,3} = \frac{31}{98} \sqrt{7},\qquad \qquad && c_{1,4} = - 
\frac{8}{108} \sqrt{7},\\
& && c_{2,2} =  \frac{1}{7} \sqrt{21}, && c_{2,3} = - \frac{41}{294} \sqrt{21},&& 
c_{2,4} = \frac{16}{315} \sqrt{21},\\
& && &&  c_{3,3} = \frac{1}{21} \sqrt{15}, && c_{3,4} = - \frac{8}{225} \sqrt{15} 
\quad \text{and}\\
& && && && c_{4,4} = \frac{2}{525} \sqrt{35};
\end{alignat*}
\end{small}	
		
	for $d=9$ we have

%\newlength{\templen}
%\setlength{\templen}{\emergencystretch}
%\setlength{\emergencystretch}{\hsize}
\begin{flushleft}
\noindent $\splitatcommas{	
{c_{1,1} = 4,}\  
{c_{1,2} = -\frac{27}{8}},\   
{c_{1,3} = \frac{115}{48}},\  
{c_{1,4} = - \frac{103}{80}},\  
{c_{1,5} = \frac{11}{28}},\ 
{c_{2,2} = \frac{1}{8} \sqrt{39}},\ 
{c_{2,3} = -\frac{925}{5616}\sqrt{39}},\  
{c_{2,4} = \frac{607}{5616} \sqrt{39}},\  
{c_{2,5} = - \frac{355}{9828} \sqrt{39}},
{c_{3,3} = \frac{1}{351} \sqrt{24882}},\  
{c_{3,4}=-\frac{34253}{8957520} \sqrt{24882}},\   
{c_{3,5} = \frac{47513}{25081056} \sqrt{638\cdot210409}},\ 
{c_{4,4} = \frac{1}{76560} \sqrt{638\cdot210409}},\  
{c_{4,5} = -\frac{586549}{4510495612} \sqrt{210409\cdot638}} \ \allowbreak
\text{and } \allowbreak
{c_{5,5} = \frac{5}{17674356} \sqrt{1262454 \cdot 2417}};}$
\end{flushleft}
			
for $d=10$ we have

\begin{flushleft}
\noindent
$\splitatcommas{
{c_{1,1} = 3},\ 
{c_{1,2} = - \frac{7}{3}},\ 
{c_{1,3} = \frac{17}{12}},\  
{c_{1,4} = - \frac{7}{15}},\ 
{c_{1,5} = \frac{71}{1512}},\ 
{c_{2,2} = \frac{1}{3} \sqrt{5}},\   
{c_{2,3} = - \frac{5}{12} \sqrt{5}},\   
{c_{2,4} = \frac{251}{1200} \sqrt{5}},\  
{c_{2,5} = - \frac{379}{30240} \sqrt{5}},\ 
{c_{3,3} = \frac{1}{4} \sqrt{2}},\ 
{c_{3,4} = - \frac{33}{160} \sqrt{2}},\  
{c_{3,5} = \frac{193}{4032} \sqrt{2}},\ 
{c_{4,4} = \frac{1}{800} \sqrt{146170}},\  
{c_{4,5} = - \frac{135673}{2016\sqrt{146170}}} \ \allowbreak \text{and } \allowbreak
{c_{5,5} = \frac{1}{504 \sqrt{14617}} \sqrt{4176691};}}$
\end{flushleft}

for $d=11$ we have

\begin{flushleft}
\noindent $\splitatcommas{
	{c_{1,1} = \sqrt{87}},\ 
	{c_{1,2}= -\frac{28}{29} \sqrt{87}},\     
	{c_{1,3}\frac{215}{261}\sqrt{87}},\    
	{c_{1,4}=-\frac{279446473}{495701640}\sqrt{87}}, \    
	{c_{1,5}=\frac{1345}{4959}\sqrt{87}},\ 
	{c_{1,6} = -\frac{22906823}{330467760}\sqrt{87}},\ 
	{c_{2,2} = \sqrt{\frac{26}{29}}},\     
	{c_{2,3}=-\frac{157}{78}\sqrt{\frac{26}{29}}},\ 
	{c_{2,4} = \frac{2674088}{1322685}\sqrt{\frac{26}{29}}},\  
	{c_{2,5}=-\frac{5485}{4446}\sqrt{\frac{26}{29}}},\ 
	{c_{2,6} = \frac{1834513}{4702880}\sqrt{\frac{26}{29}}},\ 
	{c_{3,3}=\frac{1}{3} \sqrt{\frac{467}{78}}},\  
	{c_{3,4}=-\frac{3554462909}{5321670480}\sqrt{\frac{467}{78}}},\ 
	{c_{3,5}=\frac{282916}{452523} \sqrt{\frac{467}{78}}},\ 
	{c_{3,6}=-\frac{23866165}{88694508} \sqrt{\frac{467}{78}}},\ 
	{c_{4,4}=\frac{1}{759696} \sqrt{\frac{127230362521319}{14010}}},\ 
	{c_{4,5}=-\frac{82275318718}{41095407094386037}\sqrt{\frac{127230362521319}{14010}}},\ 
	{c_{4,6}=\frac{21037454688547}{20136749476249158130}\sqrt{\frac{127230362521319}{14010}}},\ 
	{c_{5,5}=\frac{1}{2907}\sqrt{\frac{132029134219450005907}{8906125376492330}}},\ 
	{c_{5,6}=-\frac{3893456665881898045477}{10234898484691764457910640}\sqrt{\frac{132029134219450005907}{8906125376492330}}} \ \allowbreak \text{and } \allowbreak
	{c_{6,6} = \frac{1}{13328} \sqrt{\frac{110317821367843091833849}{1980437013291750088605}}};
}$
\end{flushleft}

	for $d=12$ we have
		
\begin{flushleft}\noindent$\splitatcommas{
	{c_{1,1}=\sqrt{11}}, \  
	{c_{1,2}=-\frac{35}{44} \sqrt{11}}, \ 
	{c_{1,3} = 	\frac{1}{2} \sqrt{11}}, \  
	{c_{1,4}=-\frac{655199}{2676520}  \sqrt{11}}, \  
	{c_{1,5} = \frac{1}{11}  \sqrt{11}}, \  
	{c_{1,6} = -\frac{110207}{5353040} \sqrt{11}}, 
	{c_{2,2}=\frac{1}{4}  \sqrt{\frac{95}{11}}}, \  
	{c_{2,3}=-\frac{11}{38}  	\sqrt{\frac{95}{11}}}, \  
	{c_{2,4}=\frac{147757}{660440}  	\sqrt{\frac{95}{11}}}, \  
	{c_{2,5}=-\frac{2617}{22515}  \sqrt{\frac{95}{11}}}, 	\  
	{c_{2,6}=\frac{4172389}{138692400}  \sqrt{\frac{95}{11}}}, 
	{c_{3,3}=\sqrt{\frac{1}{38}}}, \ 
	{c_{3,4} =-\frac{38345}{36498}  \sqrt{\frac{1}{38}}}, \ 
	{c_{3,5}=\frac{424}{1185}  \sqrt{\frac{1}{38}}}, \ 
	{c_{3,6}=\frac{8117}{729960}  \sqrt{\frac{1}{38}}}, 
	{c_{4,4}=\frac{1}{182490}  \sqrt{\frac{176558597}{2}}}, \ 
	{c_{4,5}=-\frac{496865}{83688774978}  \sqrt{\frac{176558597}{2}}}, \ 
	{c_{4,6}=\frac{274043219}{128880713466120}  \sqrt{\frac{176558597}{2}}},
	{c_{5,5}=\frac{1}{2370}  \sqrt{\frac{3003702364301}{2471820358}}}, \ 
	{c_{5,6}=-\frac{18498609557237}{62645216509861656} \sqrt{\frac{3003702364301}{2471820358}}} \ \allowbreak  \text{and } \allowbreak
	{c_{6,6} = \frac{1}{3080} \sqrt{\frac{1479419046289663}{2372924867797790}}};
	}$
\end{flushleft}

for $d=13$ we have
		
\begin{flushleft}\noindent $\splitatcommas{
	{c_{1,1}=4  \sqrt{\frac{41}{7}}}, \ 
	{c_{1,2}=-\frac{643}{164} \sqrt{\frac{41}{7}}}, \  
	{c_{1,3}=\frac{2285}{656}  \sqrt{\frac{41}{7}}}, \  
	{c_{1,4}=-\frac{25057169756187379}{9706601210394348}  \sqrt{\frac{41}{7}}}, \ 
	{c_{1,5}=\frac{223240166469743567}{155305619366309568}  \sqrt{\frac{41}{7}}}, \  
	{c_{1,6}=-\frac{133}{260}  \sqrt{\frac{41}{7}}}, \  
	{c_{1,7}=\frac{25638546175376171}{328531117890270240}  \sqrt{\frac{41}{7}}}, \ 
	{c_{2,2}=\frac{1}{4}  \sqrt{\frac{4423}{287}}}, \ 
	{c_{2,3}=-\frac{36921}{70768}  \sqrt{\frac{4423}{287}}}, \ 
	{c_{2,4}=\frac{648304603385756183}{1047129198867663444} \sqrt{\frac{4423}{287}}},\ 
	{c_{2,5}-\frac{8656079938938263059}{16754067181882615104}  \sqrt{\frac{4423}{287}}}, \ {c_{2,6}=\frac{336651}{1149980}  	\sqrt{\frac{4423}{287}}}, \ 
	{c_{2,7}=-\frac{3101604353052717839}{35441295961674762720}  \sqrt{\frac{4423}{287}}}, \ 
	{c_{3,3}=\frac{1}{4}  \sqrt{\frac{52445}{4423}}}, \  
	{c_{3,4}=-\frac{876978256134844544}{1671406615081576485}  \sqrt{\frac{52445}{4423}}}, \  
	{c_{3,5}=\frac{139692919173390267203}{248323268526405649200}  \sqrt{\frac{52445}{4423}}},\  {c_{3,6}=-\frac{414247923}{1172670200}  \sqrt{\frac{52445}{4423}}}, \  
	{c_{3,7}=\frac{56813568067103416}{525299221882781181}  \sqrt{\frac{52445}{4423}}},\ 	{c_{4,4}=\frac{2}{3541078198497}  \sqrt{\frac{278022966275031721230952411511}{681785}}}, \ 
	{c_{4,5}=-\frac{16842806625507324664137697618319}{15988889469513083235667255314875525722629840}
	 		\sqrt{\frac{278022966275031721230952411511}{681785}}},\ 
	{c_{4,6}=\frac{7494161056453719279901}{8607591035874982089310286660380560}  
			\sqrt{\frac{278022966275031721230952411511}{681785}}},\  
	{c_{4,7}=-\frac{115260002843756514570159114279703}{364546679904898297773213421179161986475960352}
	  		\sqrt{\frac{278022966275031721230952411511}{681785}}}, \ 
	{c_{5,5}=\frac{1}{14204782259113680}
			\sqrt{\frac{17451872904879613634434062169896872769038117850574344762823}{278022966275031721230952411511}}},\
	{c_{5,6} = -\frac{151181192623630707882481830209965662197243013451}{1404805961351189379117404268428018670416492334499832456028200208} \times}\allowbreak{
	\qquad \qquad \times \sqrt{\frac{17451872904879613634434062169896872769038117850574344762823}{278022966275031721230952411511}}},\ 
	{c_{5,7}=\frac{65110165541673840570849188883311271190521123019507743585023}{1258569508109052577461744053186022770733948051592944510245658946785111648480} \times}\allowbreak{\qquad \qquad \times
	\sqrt{\frac{17451872904879613634434062169896872769038117850574344762823}{278022966275031721230952411511}}},  \
	{c_{6,6}=\sqrt{\frac{25352491093848053461206176651408840063078575112714792491}{680623043290304931742928424625978037992486596172399445750097}}},\ 
	{c_{6,7} = 
	-\frac{116354006097743659300270866764317970081269530370009757503847822914109}{2373228356611870629873418754566765279954894894468648815804099558215400} \times}\allowbreak {\qquad \qquad \times
	\sqrt{\frac{25352491093848053461206176651408840063078575112714792491}{680623043290304931742928424625978037992486596172399445750097}}} \ 
	\allowbreak  \text{and } \allowbreak
	{c_{7,7} = \frac{1}{46804638404700} \times} \allowbreak {\qquad \qquad \times \sqrt{\frac{41971232677905854359177826832232722232187999636495894239101559440686690856571}{1064804625941618245370659419359171282649300154734021284622}}};}$
\end{flushleft}

and for $d=14$ we have	
	
\begin{flushleft}\noindent $\splitatcommas{
	{c_{1,1}=\sqrt{13}}, \  
	{c_{1,2}=-\frac{32}{39}  \sqrt{13}}, \  
	{c_{13} = \frac{7}{13}  \sqrt{13}}, \ 
	{c_{1,4}= -\frac{1581}{5005}  \sqrt{13}}, \  
	{c_{1,5}= \frac{186021079786121}{1151722559447460}  \sqrt{13}},\ 
	{c_{1,6} = -\frac{53}{897}  \sqrt{13}}, \  
	{c_{1,7} = \frac{3286004765171}{309987438360600}  \sqrt{13}	},\ 
	{c_{2,2}=\frac{2}{3}  \sqrt{\frac{17}{13}}}, \  
	{c_{2,3}=-\frac{59}{68} \sqrt{\frac{17}{13}}}, \  
	{c_{2,4}=\frac{1293}{2618}  \sqrt{\frac{17}{13}}}, \  
	{c_{2,5}=-\frac{185610347559698}{1882623414481425}  \sqrt{\frac{17}{13}}}, \ 
	{c_{2,6}=-\frac{12097}{333132}  \sqrt{\frac{17}{13}}}, \  
	{c_{2,7}=\frac{153739557217}{7871226963600}  \sqrt{\frac{17}{13}}}, \ 
	{c_{3,3}=\frac{1}{4}  \sqrt{\frac{67}{17}}}, \ 
	{c_{3,4}=-\frac{19441}{51590}  \sqrt{\frac{67}{17}}}, \  
	{c_{3,5}=\frac{646437778225307}{2473250368044225} \sqrt{\frac{67}{17}}},\ 
	{c_{3,6}=-\frac{41037}{437644}  \sqrt{\frac{67}{17}}}, \ 
	{c_{3,7}=\frac{127375771980323}{9585765401612400}  \sqrt{\frac{67}{17}}},\ 
	{c_{4,4}=\frac{1}{385}  \sqrt{\frac{59509}{67}}}, \ 
	{c_{4,5}=-\frac{56914455235667}{11576949729765700}  \sqrt{\frac{59509}{67}}},\ 
	{c_{4,6}=\frac{2624828}{680247379}  \sqrt{\frac{59509}{67}}}, \ 
	{c_{4,7}=-\frac{51012398758058}{40932786544528725}  \sqrt{\frac{59509}{67}}},\ 
	{c_{5,5}=\frac{1}{14765673839070} \sqrt{\frac{9222821132677377658501193273}{297545}}},\ 
	{c_{5,6}=-\frac{28381206876870420343}{316278205102905312042981420910989}
			\sqrt{\frac{9222821132677377658501193273}{297545}}},\ 
	{c_{5,7}=\frac{800180049472559360738859359}{21631465553921823566411244940446165459600}
	  \sqrt{\frac{9222821132677377658501193273}{297545}}},\ 
	{c_{6,6}=\frac{2}{34293}  \sqrt{\frac{32861993291352160021413495104861}{46114105663386888292505966365}}},\ 
	{c_{6,7}=-\frac{465798003932716733895505153185993057604207}{9403204725977980081537461590679792988551578400}
	\sqrt{\frac{32861993291352160021413495104861}{46114105663386888292505966365}}}
	 \ \allowbreak \text{and } \allowbreak
	{c_{7,7} = \frac{1}{95380750264800} 
	\sqrt{\frac{27167242806224591574440191807735163931063797448111661}{98585979874056480064240485314583}}}.}$
\end{flushleft}
	
\end{corollary}

Due to the fact that we were able to find a feasible solution to the SDP 
derived in Observation~\ref{obs:SDPforCert} for any $d\leq 14$ we have 
the following conjecture. 

\begin{conjecture}
	Let $k_\G = k_\H = 1$ and let $d = n_\G + n_\H - 1$. Then a 
	$\lceil d/2 \rceil$-SOS-certificate of $f_\viz$ of the form presented in 
	Theorem~\ref{thm:general-cert} exists, as there is a positive 
	semidefinite matrix $F$ fulfilling the system of 
	equations~\eqref{eq:sysOfEquInF} in Observation~\ref{obs:SDPforCert}.
\end{conjecture}

Concerning the value of $F_{1,1}$, we know from Corollary~\ref{cor:fixedF11} that $F_{1,1} = d - 1$  for even $d$, for odd $d$ we make the following observation.
\begin{observation}
	For the certificates above with $d \leq 13$ and $d$ odd it turns out that $F_{1,1}$ is not fixed. Moreover, for these certificates the choice of $F_{1,1} = d - 1$ is not possible.
\end{observation}

For $d > 14$  
we did not derive certificates because of  numerical difficulties with 
off-the-shelf SDP solvers.

	\section{Conclusion and Open Questions}
\label{chapter:conclusion}

%Summary, conclusion and outlook
In this paper, we extended the approach of Gaar et 
al.~\cite{vizing-short2019,vizing-long2020} to prove Vizing's conjecture via an 
algebraic method for 
graph classes $\G$ and $\H$, where the graph classes~$\G$ and~$\H$ are defined 
as all graphs with~$n_\G$ and~$n_\H$ vertices and a minimum 
dominating set of size~$k_\G$ and~$k_\H$, respectively.
We applied their technique to the case where both minimum dominating sets 
in~$\G$ and $\H$ are of size~1.
A bottleneck in their computations is the time-consuming intermediate step to determine a Gröbner basis of~$I_\viz$.
We were able to overcome this obstacle by determining the unique reduced Gröbner basis of $I_\viz$ for this case.
This allowed us to conclude that if an $\ell$-SOS-certificate exists, 
it must be at least of degree~$\ell = \lceil (n_\G + n_\H - 1)/2 \rceil$.

We further presented a procedure to find~$\lceil (n_\G + n_\H - 1)/2 
\rceil$-SOS-certificates of a special form for Vizing's conjecture on these 
graph classes~$\G$ and~$\H$.
This new approach is based on our knowledge of the Gröbner basis, and assumes 
that in addition to the polynomials of degree 2, only one polynomial of higher 
degree of the Gröbner basis is sufficient to prove correctness of the 
SOS-certificate.
Assuming a specific form of the SOS-certificate, the coefficients of the 
polynomials of this certificate can be obtained by solving a system of 
quadratic equations. We presented a method how to obtain an exact solution to 
this using SDPs, that avoids clever guessing as usually needed in the approach 
from~\cite{vizing-short2019,vizing-long2020}.
The specific form of the certificates yields that certificates of classes with $n_\G + n_\H - 1 = d$ depend only on~$d$ and not on~$n_\G$ or~$n_\H$. 
We implemented this new method in SageMath~\cite{sagemath} and used it to 
find certificates for all graph classes $\G$ and $\H$ 
with~$n_\G + n_\H \leq 15$ and domination numbers~$k_\G = k_\H = 
1$.
Even though this does not advance what is known with respect to Vizings's 
conjecture, deriving this new certificates is an important step in the area of 
using conic linear optimization for computer-assisted proofs because it 
demonstrates that deriving such proofs is possible. 

We were not able to derive certificates for 
$n_\G + n_\H > 15$ due to numerical difficulties with 
off-the-shelf SDP solvers. This needs to be dealt with in more detail.
In future work, another topic for further investigation is whether the system 
of linear equation, which has to be solved in our new approach, is solvable for 
any~size $d = n_\G + n_\H - 1$.

Most of all, the question of a general certificate depending on the size~$d$ arises.
We know that~$c_{1,1} = \sqrt{n_\G + n_\H - 2}$ holds in the 
case of odd~$n_\G + n_\H$.
For~$n_\G + n_\H$ even, however, this is not the case.
This coefficient as well as all other coefficients  are among the most obvious 
future topics to work on to find a general certificate.

With our work, we know there are SOS-certificates for Vizing's conjecture for all 
graph classes $\G$ and $\H$ with 
$k_\G = n_\G - 1 \geq 1$ and~$k_\H 
= n_\H - 1$ for~$n_\H \in \{2,3\}$, with~$k_\G = n_\G$ 
and~$k_\H = n_\H - d$ for~$d \leq 4$, 
and now also with  
$k_\G = k_\H = 1$ and $n_\G + n_\H \leq 15$.
Clearly, it would be interesting to derive SOS-certificates also for other graph 
classes $\G$ and $\H$.

	% literature
	\bibliographystyle{abbrv}
	\bibliography{thesis}

\ifJSC

\clearpage
\section*{Declarations}
\subsection*{Declaration of interests}
The authors declare that they have no known competing financial interests or 
personal relationships that could have appeared to influence the work reported 
in this paper.

\subsection*{Color of figures in print}
There is no need to use color for any of the figures in print.

\fi

\end{document}